%% file: space_of_generators.tex
\documentclass[a4paper]{amsart}

\input{amsart_header_18_05}

\usepackage[foot]{amsaddr}

\numberwithin{equation}{section} 

\usepackage{enumitem} 

\newcommand{\M}[1][n]{\mathrm{M}_{#1}}
\newcommand{\PGL}{\mathrm{PGL}_n}
\newcommand{\Ob}{\mathrm{O}^\mathrm{b}_n}
\newcommand{\POb}{\mathrm{PO}^\mathrm{b}_n}
\newcommand{\Spn}{\mathrm{Sp}_n}
\newcommand{\PSp}{\mathrm{PSp}_n}
\newcommand{\PS}{\mathrm{P}S}
\DeclareMathOperator{\gen}{gen}

\title[Spaces of Generators]{Spaces of Generators for Azumaya Algebras with Unitary Involution}

\author{Omer Cantor$^*$}
\email{ocantor@proton.me}

\author{Uriya A.\ First$^*$}
\address{$^*$University of Haifa}
\email{uriya.first@gmail.com}

\begin{document}

\begin{abstract}
Let $A$ be a finite dimensional algebra  (possibly
with some extra structure) over an infinite field
$K$ and let $r\in\N$.
The $r$-tuples $(a_1,\dots,a_r)\in A^r$
which fail to generate $A$ are the $K$-points of a closed 
subvariety $Z_r$ of the affine space underlying $A^r$, the codimension of which may be thought of
as quantifying how well a generic $r$-tuple in $A^r$ generates $A$. 
Taking this intuition one step further,
the second author, Reichstein and Williams  showed   that lower bounds
on the codimension of $Z_r$ in $A^r$ (for every $r$) imply
upper bounds on the number of generators of \emph{forms} of the $K$-algebra $A$ over
finitely generated $K$-rings. That work also demonstrates how finer
information on $Z_r$ may be used to construct forms of $A$ which require many elements
to generate.

The dimension and irreducible components of $Z_r$ are known
in a few cases, which in particular lead to upper bounds on the number
of generators of Azumaya algebras and Azumaya algebras with involution of the first
kind (orthogonal or symplectic). This paper treats 
the case of Azumaya algebras with a  unitary
involution by finding the dimension and irreducible components of $Z_r$
when $A$ is the $K$-algebra with involution $(\nMat{K}{n}\times \nMat{K}{n},
(a,b)\mapsto (b^\trans,a^\trans))$. 
Our analysis implies
that
every Azumaya algebra with a unitary involution over a finitely generated $K$-ring
of Krull dimension $d$
can be generated by $\floor{\frac{d}{2n-2}+\frac{3}{2}}$ elements.
We also give examples which require at least half that many elements
to generate, by building on
the work of the second author, Reichstein and Williams.

Our method of finding the dimension and irreducible components
of $Z_r$ actually applies to all $K$-algebras $A$ satisfying a mild assumption.

\end{abstract}

\maketitle

\section{Introduction}


Let $K$ be an infinite field and let $A$ be a finite-dimensional $K$-algebra, 
possibly
with additional structure such as an involution. 
Let $R$ be a $K$-ring,
i.e., a commutative   $K$-algebra.
By an \emph{$R$-form} of $A$, we mean an $R$-algebra $B$
for which there exists  
a faithfully flat $R$-ring $S$ such that $A_S:=A\otimes_KS$ is isomorphic
to $B\otimes_R S$ as $S$-algebras (with additional structure, if any).

For this introduction, we will focus on the following   algebras and algebras-with-involution over $K$:
\begin{enumerate}[label=(\roman*)]
\item $A=\nMat{K}{n}$,
\item $A=(\nMat{K}{n},\trans)$, where $\trans$ is matrix transposition,
\item $A=(\nMat{K}{n},
[\begin{smallmatrix} a & b \\ c & d \end{smallmatrix}]
\mapsto 
[\begin{smallmatrix} d^\trans & -b^\trans \\ -c^\trans & a^\trans \end{smallmatrix}]
$), where $n$ is even
and $a,b,c,d\in\nMat{K}{n/2}$,
\item $A=(\nMat{K}{n}\times \nMat{K}{n},(a,b)\mapsto (b^\trans,a^\trans))$.
\end{enumerate}
For a $K$-field $F$, it is well-known that in case (i), the $F$-forms of $A$
are the  central simple $F$-algebras of degree $n$,
and in cases (ii), (iii) and (iv),
the $F$-forms are the 
degree-$n$ central simple $F$-algebras with involution 
of orthogonal, symplectic or unitary type, respectively; see \cite{Knus_1998_book_of_involutions}, for instance. 
This  fact is fundamental to the study and classification of simple algebraic groups of classical
type.
Over a $K$-ring $R$, the $R$-forms of $A$ are the degree-$n$ Azumaya $R$-algebras 
in case (i), and the degree-$n$ Azumaya $R$-algebras carrying an involution of orthogonal, symplectic or unitary
type in cases (ii), (iii) or (iv), respectively; see \cite[III.8.5]{Knus_1991_quadratic_hermitian_forms},
\cite[\S1.4]{First_2022_octagon} 
and Proposition~\ref{PR:forms-of-A} below.

For every $r\in \N$, let $V_r$
denote the $K$-affine space 
underlying $A^r$. In \cite[\S2]{First_2022_generators_of_alg_over_comm_ring},
it was observed that $V_r$ has a closed subscheme $Z_r$
of ``$r$-tuples not generating $A$'', that is, 
for every $K$-field $F$, the $F$-points of $Z_r$
are the   $r$-tuples   $(a_1,\dots,a_r)\in A^r_F$ which do not generate
$A_F$ as an $F$-algebra (with extra structure, if any).


On its face,
the codimension of $Z_r$ in $V_r$ quantifies
how likely it is for   an $r$-tuple
in $A^r$ to generate $A$. 
However, in \cite{First_2022_generators_of_alg_over_comm_ring},
the second author, Z.\ Reichstein and B.\ Williams observed that
$Z_r$ is strongly related to the number of generators of $R$-forms of $A$.
This manifests in two ways: First, a lower
bound on the codimension of $Z_r$ in $V_r$ (for every $r$)
implies  an upper bound on the number of generators of any $R$-form of $A$ depending
only on the Krull dimension of $R$ \cite[Thm.~6.1]{First_2022_generators_of_alg_over_comm_ring}.
Second, finer information on how $Z_r$
sits inside $V_r$ may be used in an ad hoc manner to produce forms of $A$ which
require many generators; see 
\cite[Thm.~1.5(b)]{First_2022_generators_of_alg_over_comm_ring}, \cite{Gant_2024_space_gens_2by2_matrix},
\cite{Shukla_2021_cassifying_spaces_for_etale_algebras}.

At the basis of these results lies a deeper observation, which also serves
as a motivation to  study $Z_r$.
Let $G$ be the $K$-algebraic group of automorphisms of $A$
and let $U_r=V_r-Z_r$.
The evident action of $G$ on $V_r$
restricts to a free action   on $U_r $, and the $G$-torsor $U_r\to Y_r:=U_r/G$
corresponds to a $Y_r$-form of $A$ 
(consisting of a sheaf of $\calO_{Y_r}$-algebras).
It turns out that this $Y_r$-form 
carries a canonical generating $r$-tuple (of global sections)
and it is universal among the forms of $A$ equipped with a generating $r$-tuple 
\cite[Prop.~4.1]{First_2022_generators_of_alg_over_comm_ring}.
Since the classifying stack of $G$, denoted $BG$, classifies
forms of $A$, there is a canonical morphism $Y_r\to BG$ corresponding to forgetting the generating
$r$-tuple.
Finding the number of generators of an $R$-form $B$ may therefore be restated as the problem
of determining whether the morphism $\Spec R\to BG$ corresponding to $B$ may be lifted
to a morphism $\Spec R\to Y_r$ (corresponding to specifying $r$ generators for $B$).
Less formally, we wish to know ``how well'' does   the $K$-algebraic space $Y_r=(V_r-Z_r)/G$
approximates $B G$.
To see how this relates to $Z_r$, note for example  that if the codimension of $Z_r$ in $V_r$ is $m$,
then $Y_r\to BG$   induces an isomorphism   $\HH^{i}(BG)\to \HH^{i}(Y_r)$ when $i<2m-1$
for  some ``nice'' cohomology theories, e.g., singular cohomology when $K=\C$
\cite[Lem.~11.2, Rmk.~1.7]{First_2022_generators_of_alg_over_comm_ring}.
Thus, if  we can guarantee   that $m=\dim V_r-\dim Z_r$ is large, we may interchange $BG$
and $Y_r$ in some cohomological computations. 
On the other hand, information about when    $\HH^i(BG)\to \HH^i(Y_r)$ is  not an isomorphism
may be used to obstruct lifts along $Y_r\to BG$ \cite[Thm.~11.4]{First_2022_generators_of_alg_over_comm_ring}, giving rise to forms of $A$ which
cannot be generated by $r$ elements. Beyond what is guaranteed by dimension considerations,
such information is usually obtained by
careful case-specific analysis of the pair $Z_r\subseteq V_r$.

\medskip

So far, the dimension and irreducible components of $Z_r$
were determined 
in cases  (i)--(iii) and for a few more $K$-algebras.
Specifically, the cases where $A=\nMat{K}{n}$
and $A$ is a split octonion $K$-algebra 
were treated in \cite{First_2022_generators_of_alg_over_comm_ring},
and cases (ii) and (iii) above were resolved in \cite{Nam_2023_generators_for_matrix_algebra_with_inv}.
This  led in particular  to upper bounds
on the number of generators of Azumaya algebras  
and Azumaya algebras carrying an orthogonal or symplectic involution
listed in \cite[Thm.~1.5(a), Prop.~7.3]{First_2022_generators_of_alg_over_comm_ring}.
In addition, a more careful study of $Z_r$ and $U_r=V_r-Z_r$
was done for  $A=K\times\dots\times K$ ($n$ times) --- in which case
the forms of $A$ are the rank-$n$ finite \'etale algebras ---
and  $A=\nMat{K}{n}$, leading 
to examples of  finite \'etale algebras and
degree-$n$ Azumaya algebras which require many elements to generate;
see \cite{Shukla_2021_cassifying_spaces_for_etale_algebras} for
$A=K\times\dots\times K$,
\cite[Thm.~1.5(b)]{First_2022_generators_of_alg_over_comm_ring}
for $A=\nMat{K}{n}$ with $n>2$
and \cite{Gant_2024_space_gens_2by2_matrix} for $A=\nMat{K}{2}$.

In this paper, we add to this picture by finding the dimension 
and irreducible components of $Z_r$ in the ``unitary'' case (iv). We show:

\begin{thm}\label{TH:codim-unitary}
Suppose $A=(\nMat{K}{n}\times \nMat{K}{n},(a,b)\mapsto (b^\trans,a^\trans))$ and let $r\in\N$.
Then
\[\dim Z_r = 
\left\{\begin{array}{ll}
2rn^2 - (2r-1)(n-1) & n>1 \\
r & n=1.
\end{array}\right.
\]
\end{thm}

We also determine   the irreducible components of
$Z_r$ in Theorem~\ref{TH:irred-comps}; it has  $n$ or $n+1$ irreducible components
depending on the values of $n$, $r$ and $\Char K$.

In addition,  our Proposition~\ref{PR:forms-of-A} says
that the $R$-forms of $A$ from    (iv)
are the Azumaya $R$-algebras with unitary involution (in the sense of 
\cite[\S1.4]{First_2022_octagon}). Thus, our computation of
$\dim Z_r$ and \cite[Thm.~6.1]{First_2022_generators_of_alg_over_comm_ring}
implies: 

\begin{cor}\label{CR:generators-of-Azumaya-alg}
	Let $R$ be a ring of Krull dimension $d$  finitely generated over an infinite field $K$,
	let $n>1$ and let $(B,\tau)$ be a degree-$n$ Azumaya algebra with a unitary involution  over $R$.
	Then
	\[
	\gen_R(B,\tau)\leq \Floor{\frac{d}{2n-2}+\frac{3}{2}}.
	\]
\end{cor}

By building on \cite[Thm.~1.5(b)]{First_2022_generators_of_alg_over_comm_ring}
and \cite{Gant_2024_space_gens_2by2_matrix},
which give  examples   of degree-$n$ Azumaya algebras (without involution)
over rings of Krull dimension $d$ requiring slightly more than $ \frac{d}{2n-2} $
elements to generate,  we also observe that there exist
degree-$n$ Azumaya algebras with unitary involution  over finitely generated $K$-rings
which cannot be generated by less than  $ \frac{d}{4n-4 }  $ elements, provided $\Char K=0$;
see Remark~\ref{RM:lower-bounds}.

\medskip

Our approach to finding the dimension and irreducible components of $Z_r$
in case (iv) relies on the following result, which gives a method
to find them for any $K$-algebra $A$ (or even a \emph{multialgebra},
see Section~\ref{sec:non-generators}) satisfying a mild   assumption.

\begin{thm}\label{TH:dimZr-and-irred-comps}
	Suppose $K$ is algebraically closed. Let $G$ be the identity connected component
	of the $K$-algebraic group $\Aut_K(A)$, and
	suppose that the action of $G(K)$ on the maximal subalgebras
	of $A$ has finitely many orbits represented by $A_1,\dots,A_t$.
	For   $i\in\{1,\dots,t\}$,
	let $X_i$ be the Zariski closure of
	$\bigcup_{g\in G(K)}g(A_i)^r$ in $V_r$. Then: 
	\begin{enumerate}[label=(\roman*)]
		\item The irreducible components of $Z_r$
		are the maximal members of the family $X_1,\dots,X_t$.
		\item
		$\dim Z_r \leq  \dim G+\max\{r\dim A_i- \dim \Stab_{G}(A_i)\where i\in\{1,\dots,t\}\}$,
		and equality holds if the maximum is attained at an index $i$ for which $A_i$
		can be generated by $r$ elements.
	\end{enumerate}
\end{thm}

Since tensoring with the algebraic closure $\overline{K}$
does not affect $\dim Z_r$, this implies: 

\begin{cor}
	Let $A$ be a finite-dimensional $K$-algebra such that the $\overline{K}$-algebra
	$A\otimes_K \overline{K}$
	satisfies the assumption of Theorem~\ref{TH:dimZr-and-irred-comps}.
	Then there exist  integers $a,b\geq 0$ such that for every sufficiently
	large $r$, we have $\dim Z_r = ar+b$. 
\end{cor} 

Suppose now that $K$ is algebraically closed and $A$ is the algebra with involution
$(\nMat{K}{n}\times \nMat{K}{n},(a,b)\mapsto (b^\trans,a^\trans))$ from (iv).
Thanks to Theorem~\ref{TH:dimZr-and-irred-comps}, the task of finding $\dim Z_r$
breaks down to finding  the maximal subalgebras of $A$ (invariant under the involution)
up to the action of $G(K)$
($G$ as in the theorem), finding the       stabilizer  in $G(K)$ of each representative,
checking that every representative can   be generated by $r$ elements,
and dealing separately with any  exceptions to the latter. 
This is carried out in Section~\ref{sec:maximal-subalg}.

Determining the irreducible components is trickier, because  contrary to what one may expect,
the family $X_1,\dots,X_t$ from Theorem~\ref{TH:dimZr-and-irred-comps} 
may include non-maximal members, which indeed happens in our case;
see Lemma~\ref{LM:do-contain}.
Another difficulty  arises from the fact that
$\bigcup_{g\in G(K)}g(A_i)^r$ is not always  Zariski closed in $A^r=V_r(K)$ (Remark~\ref{RM:not-closed}).
Our Theorem~\ref{TH:irred-comps} determines the maximal $X_i$ for our $A$,
showing in particular that
the number of irreducible components is the expected number $t$, except
when $n$ is even and $\Char K=2$, or $(n,r)=(2,1)$.

\medskip

The paper is organized as follows: Section~\ref{sec:conventions}
sets notation and conventions for the paper.
Section~\ref{sec:non-generators} recalls some facts
about $Z_r$ and proves Theorem~\ref{TH:dimZr-and-irred-comps}.
In Section~\ref{sec:Az-alg-with-inv}, we recall Azumaya algebras
with unitary involution, and prove that the $R$-forms of
the algebra with involution $A$ in (iv) are precisely
the Azumaya $R$-algebras with unitary involution.
Section~\ref{sec:maximal-subalg} contains
results about the maximal subalgebras of $A$ from (iv)
that are needed in order to apply Theorem~\ref{TH:dimZr-and-irred-comps}.
We then use Theorem~\ref{TH:dimZr-and-irred-comps} to prove
Theorem~\ref{TH:codim-unitary} and Corollary~\ref{CR:generators-of-Azumaya-alg}
in Section~\ref{sec:dim-Zr}.
Finally, in Section~\ref{sec:irred-comps}, we use
Theorem~\ref{TH:dimZr-and-irred-comps}  to find the irreducible components of
$Z_r$.

\subsection*{Acknowledgements.}
We are grateful to Z.~Reichstein and B.~Williams for comments
on an earlier version of this manuscript. 
The second author is supported by an ISF grant no.\ 721/24.

\section{Conventions}
\label{sec:conventions}

Throughout, a ring means a commutative (associative) unital ring.
By default, algebras are associative and unital, and may
carry an involution when indicated. Subalgebras
are required to include the unity and be stable under the involution, if present. 
Nevertheless, in Section~\ref{sec:non-generators}
we will allow non-unital non-associative algebras.
Given a ring $R$, an $R$-ring means a commutative (associative, unital) $R$-algebra.
The group of invertible elements in $R$ is denoted
$R^\times$.

We always let $K$ denote a field. Elements of $K^n$
are   viewed as column vectors. 
The transpose of an $m\times n$ matrix $a$ is denoted $a^\trans$.
When there is no risk of confusion, we shall write
$\M$ for $\nMat{K}{n}$ and $\GL_n$ for $\nGL{K}{n}$.
The $n\times n$ identity matrix is denoted $I_n$, or just $I$.
When $n$ is even, we also let $\Omega=\Omega_n:=I_{n/2}\otimes[\begin{smallmatrix}0&-1\\1&0\end{smallmatrix}]$.

Recall that two matrices $a,a'\in\M$
are called congruent, denoted $a\sim a'$, if there is $g\in\GL_n$
such that $gag^\trans=a'$. This is equivalent to saying that the bilinear forms
$b(x,y)= x^\trans a y$ and $b'(x,y)=x^\trans a' y$ on $K^n$ are isomorphic.
Let $a\in\GL_n$ and suppose $K$
is algebraically closed.
Then it is known
(\cite[Props.~1.9, 1.18]{Elman_2008_algebraic_geometric_theory_of_quad_forms}, for instance)
that $a$ is symmetric and not alternating if and only if $a\sim I$ 
and $a$ is alternating if and only if $n$ is even and
$a\sim \Omega$.
(Here, as usual, $a$ is alternating if it is skew-symmetric and has zeroes on
the diagonal; the latter condition is needed only when $\Char K=2$.)

A $K$-variety is a reduced $K$-scheme of finite type, not necessarily irreducible.
A $K$-algebraic group is a group scheme of finite type over $K$, possibly not smooth.
When $K$ is algebraically closed, we will freely identify
any $K$-variety (resp.\ $K$-algebraic group) with its underlying set (resp.\ group) of $K$-points, to which
we give the Zariski topology. We will also define morphisms
between $K$-varieties
by specifying them as functions between the corresponding sets of $K$-points.
We shall view every finite-dimensional $K$-vector space 
as a $K$-affine space. In particular, such spaces will be endowed with the Zariski
topology. The $K$-projective space associated to a finite dimensional $K$-vector space
$V$ is denoted $\bbP(V)$.

\section{Spaces of Non-Generating Tuples}
\label{sec:non-generators}

Let us fix a finite-dimensional $K$-multialgebra $A$.
Recall that this means that $A$ is a $K$-vector space equipped with some $K$-multilinear
operations $\{\mu_i:A^{n_i}\to A\}_{i\in I}$. Typically, $A$ will be an ``ordinary'' $K$-algebra
--- the product being a $K$-bilinear map $\mu:A\times A\to A$ --- carrying some additional structure
such as an involution (a $K$-linear map $\mu:A\to A$) or a unit  element  (encoded as a multilinear map with
$0$ arguments $\mu: A^0\to A$). 
The $K$-algebraic group of $K$-multialgebra automorphisms of $A$
is denoted $\Aut_K(A)$.
As usual, the sub-multialgebra of $A$ generated by a subset
$S\subseteq A$ is the smallest subspace of $A$ that is closed under all the operations
of $A$. We write
\[
\gen(A)\qquad\text{or} \qquad \gen_K(A)
\]
for the smallest cardinality of a generating subset for $A$.
For a $K$-ring $R$, we let $A_R=A\otimes_K R$, which is a multialgebra over $R$.
One may now consider $R$-forms of $A$ as in the Introduction.

As before,
given $r\in\N$, we write $V_r$ for the $K$-affine space underlying $A^r$. 
In \cite{First_2022_generators_of_alg_over_comm_ring}, the second author, Reichstein and Williams
defined a closed subscheme $Z_r$ of $V_r$ such that for every $K$-field $L$, we have
\[
Z_r(L)=\{(a_1,\dots,a_r)\in A_L^r\suchthat
\text{$a_1,\dots,a_r$ do not generate $A_L$}\}.
\]
The codimension of $Z_r$ in the ambient variety $V_r$ is  
\[c_A(r):=r\dim A -\dim Z_r.\]
The following result from
\cite{First_2022_generators_of_alg_over_comm_ring}  says that bounding $c_A(r)$ from below 
provides an upper bound
for the number of generators of forms of $A$.  

\begin{thm}[{\cite[Thm.~6.1]{First_2022_generators_of_alg_over_comm_ring}}]
	\label{TH:cAr-and-number-of-generators}
	Let $d,r\in\N\cup\{0\}$. If $c_A(r)>d$, then every form of $A$ over a finitely generated
	$K$-ring $R$ of Krull dimension $d$ or less can be generated by $r$ elements as an $R$-multialgebra.
\end{thm}

The formation of $Z_r$ commutes with base change, so its
dimension does not change if we extend scalars from $K$ to its algebraic closure $\overline{K}$.
We  may therefore restrict our attention to the case   $K=\overline{K}$.
Provided this holds, we will abuse the notation and use $Z_r$ to indicate its
set of $K$-points
as we do for   $K$-varieties.

Suppose that $K$ is algebraically closed and let $m$ denote the maximal dimension of
a proper subalgebra of $A$.
By \cite[Lem.~6.2]{First_2022_generators_of_alg_over_comm_ring}, we 
have $mr \leq \dim Z_r \leq mr+m(\dim A-m)$ for every
$r > m$. 
Our first result gives a recipe to find the exact dimension of $Z_r$ and also its  irreducible components; 
it readily implies  Theorem~\ref{TH:dimZr-and-irred-comps}.

\begin{prp}\label{PR:dim-Zr-detailed}
With notation as above,
assume that $K$ is algebraically closed and let $G$ be a closed
subgroup of   $\Aut_K(A)$ (more precisely,
its group of $K$-points).
Suppose that the action of $G $ on the set of maximal $K$-subalgebras of $A$ has finitely many orbits represented by
subalgebras $A_1,\dots,A_t$. For every $i\in\{1,\dots,t\}$,
let $H_i$ denote the stabilizer of $A_i$ in $G$. In addition, fixing  $r\in\N$,
let
\[
Y_i = \bigcup_{g\in G} g(A_i)^r
\qquad\text{and}\qquad
X_i = \overline{Y_i},
\]
where the closure is taken in  $A^r$ w.r.t.\ Zariski topology. Then:
\begin{enumerate}[label=(\roman*)]
	\item $Z_r=\bigcup_{i=1}^t X_i$.
	\item $\dim X_i \leq  \dim G + r\dim A_i-\dim H_i$, and if $\gen(A_i)\leq r$, equality holds.
	\item If $G$ is connected, then the irreducible components of
	$Z_r$ are the maximal members of the family  $X_1,\dots,X_t$.
	\item If   $\gen(A_i)\leq r$ and $Y_j$ is closed in $A^r$ for some $i\neq j$, 
	then $X_i\nsubseteq X_j$.
	\item If $H_i$ is a parabolic subgroup of $G$, then $Y_i$ is closed in $A^r$, i.e.,
	$X_i=Y_i$.
\end{enumerate}
\end{prp}

We will usually apply the proposition with $G$ being the identity
connected component of $\Aut_K(A)$.

\begin{proof}
Observe that $G$ acts on $A^r$ from the left by $g(a_1,\dots,a_r)=(g(a_1),\dots,g(a_r))$.

(i) It is enough to show that $Z_r=\bigcup_{i=1}^t Y_i$,
since that would imply $Z_r=\overline{Z_r}=\bigcup_{i=1}^t \overline{Y_i}=\bigcup_{i=1}^t X_i$.
Every $r$-tuple $a=(a_1,\dots,a_r)\in Y_i$ is contained in $g(A_i)^r$ for some $g\in G$,
and thus generates a subalgebra of $g(A_i)$. In particular, $a$ cannot generate $A$,
so $a\in Z_r$.
Conversely, any tuple $a=(a_1,\dots,a_r)\in Z_r$ generates a proper subalgebra of $A$,
which is contained in $g(A_i)$ for some $i$ and $g\in G$, so $a\in \bigcup_{i=1}^t Y_i$.

(ii) Define a morphism $\vphi_i : G\times A_i^r\to X_i$ by $\vphi_i(g,a)=
g(a)$. The image of $\vphi_i$ is precisely $Y_i$, which is dense in $X_i$.
In view of the fiber dimension theorem, in order
to prove
that $\dim X_i \leq  \dim G + r\dim A_i-\dim H_i$,
it is enough to show that
\begin{equation} \label{EQ:dim-im}
\dim(\vphi^{-1}(a)) \geq \dim(H_i) , 
\end{equation}
where $a = (a_1, \dots, a_r) \in X_i$ is a point in general position in $X_i$
(i.e., $a$ ranges in a dense subset of $X_i$). We also need to show  
that equality holds when $A_i$ can be generated by $r$ elements.

We may assume that $a\in Y_i$.
Moreover, since $\varphi_i$ is $G$-equivariant,
 we may 
replace $a$ with
$g(a)$ for a suitable
$g\in G$ and assume without loss of generality that $a\in A_i^r$.
Define a morphism
$\psi:H_i\to \vphi_i^{-1}(a)$
by $\psi(h)=(h^{-1},h(a))$ (we have $h(a)\in A^r_i$ because $H_i$ stabilizes $A_i$ and $a\in A_i^r$). Then $\psi$ is one-to-one and hence \eqref{EQ:dim-im} holds. 

Suppose now that $A_i$ can be generated by $r$ elements.
Arguing as before, we again assume that $a\in A_i^r$.
However, we are now allowed to further assume that ${a}$ generates $A_i$. 
We claim that $\psi:H_i\to \vphi_i^{-1}(a)$
is also surjective, and thus $\dim \vphi^{-1}_i(a)=\dim H_i$.
Indeed, let $(g,x)\in \vphi^{-1}(a)$. Then $ x=g^{-1}(a)$. Since
the tuple $a$ generates $A_i$ and $x\in A_i^r$, this means that 
$g^{-1}(A_i)\subseteq A_i$, so $g^{-1}\in H_i$
and $\psi(g^{-1})=(g,g^{-1}(a))=(g,x)$. 

(iii) Since $G$ is a connected algebraic group, it is irreducible.
Thus,  $G\times A_i^r$ is irreducible 
and   so is   $X_i = \overline{\vphi_i(G\times A_i^r)}$.
Since $Z_r=\bigcup_{i=1}^t X_i$, the irreducible components
of $Z_r$ are the maximal members of $X_1,\dots,X_t$.

(iv) 
	Let $U_i$ denote the set of $r$-tuples in $A_i$ which generate $A_i$ as a $K$-algebra.
	Our assumptions imply that $U_i\neq\emptyset$. 
	Since $U_i\subseteq X_i$ and $Y_j=X_j$, in order to show that $X_i\nsubseteq X_j$,
	it is enough to show that $U_i\cap Y_j=\emptyset$.
	For the sake of contradiction, suppose 
	that there is $a\in U_i\cap Y_j$. Then there is some $g\in G $ such that
	$a\in g (A_j)^{r}$.
	Since the tuple $a$ generates $A_i$, it follows that $A_i\subseteq g(A_j)$.
	The maximality of $A_i$ and $A_j$ now forces $A_i=g(A_j)$, but this contradicts our
	assumption that $A_i$ and $A_j$ are not in the same $G $-orbit. Thus $a$ cannot
	exist and $U_i\cap Y_j=\emptyset$.

(v) Recall that $H_i$ is parabolic if $G/H_i$ is a complete $K$-variety.
Write $m=\dim A_i$ 
and let $\operatorname{Gr}(m,A)$ denote the Grassmannian variety of $m$-dimensional
subspaces of  $A$. There is a morphism $\psi: G/H_i\to \operatorname{Gr}(m,A)$
mapping the coset $gH_i$ to $g(A_i)$. Since $G/H_i$ is complete, its image $Z$ is a closed
subvariety of $\operatorname{Gr}(m,A)$, hence complete. 
Consider the closed subvariety $W$ of $Z\times A^r$ consisting
of tuples $(V,a_1,\dots,a_r)$ with $a_1,\dots,a_r\in V$.
The image of $W$ under the second projection $p_2: Z\times A^r\to A^r$
is precisely $Y_i$. Since $Z$ is complete, $p_2$ is closed, and we conclude that $Y_i=p_2(W)$
is closed in $A^r$. 
\end{proof}

\begin{remark}
	With notation
	as in Proposition~\ref{PR:dim-Zr-detailed}, 
	it may happen that $X_i\subsetneq X_j$ for some $i\neq j$ even when $G$ is connected. 
	In addition, is may also happen that 
	$Y_i$ is not closed, i.e., $Y_i\neq X_i$.
	This is demonstrated in Lemma~\ref{LM:do-contain}
	and Remark~\ref{RM:not-closed} below, and shows that the
	assumptions   in (iv) and   (v) cannot be removed in general.
\end{remark}

\begin{remark}
	It is in general not true that the action of (the $K$-points of) 
	$\Aut_K(A)$ on the maximal subalgebras
	of $A$ has finitely many orbits. Here is such an example with $A$ commutative, associative
	and unital: Let $f\in K[x_0,x_1,x_2]$ be a homogeneous polynomial such that
	the equation $f=0$ defines a genus-$3$ curve in $\bbP^2_{K}$
	having no nontrivial automorphisms; such $f$ exists, e.g., by \cite{Poonen_2000_varieties_no_aut}.
	Write $m=\deg f>1$ and
	let $A=K[x_0,x_1,x_2]/((x_0,x_1,x_2)^{m+1}+(f))$. Then $A$ is a local $K$-algebra
	with Jacobson radical $J=x_0A+x_1A+x_2A$. Let $\quo{x_i}$ denote the image
	of $x_i$ in $J/J^2$. Then every automorphism $g$ of $A$ induces an element
	$\quo{g}\in \GL(J/J^2)$ which must satisfy $f(\quo{g}(\quo{x_0}),
	\quo{g}(\quo{x_1}),\quo{g}(\quo{x_2}))\in Kf(\quo{x}_0,\quo{x}_1,\quo{x}_2)$.
	Thus, $\quo{g}$ induces an automorphism of the projective curve $f=0$, and   must therefore
	be scalar. As a result, if $U$  is a codimension-$1$ 
	subspace of $J$ containing $J^2$, then $g(U)=U$.
	For every such $U$, the space  $K+U$ is a maximal subalgebra of $A$
	fixed by $\Aut_K(A)$, and there are infinitely many such subalgebras since
	$K$ is infinite.
\end{remark}

%
%

The following proposition gives a more sophisticated way
for checking that $X_i\nsubseteq X_j$ when $\dim A_i=\dim A_j$. It will be needed in
Section~\ref{sec:irred-comps}.

\begin{prp}\label{PR:Grassmannian-test}
	Assume that $K$ is algebraically closed, let $A$ be a $K$-multialgebra and let
	$G$ be a closed subgroup of $\Aut_K(A)$. Let $A_1$ and $A_2$ be two subalgebras of $A$
	having the same dimension $m$.
	Let $X$ denote the Grassmannian of $m$-spaces inside $A$ and,  for $i=1,2$, let
	$Z_i=\{g(A_i)\where g\in G\}\subseteq X$.
	If $\overline{Z_1}\nsubseteq \overline{Z_2}$
	(the closures are taken in $X$), then $\overline{\bigcup_{g\in G}g(A_1)^r}\nsubseteq
	\overline{\bigcup_{g\in G}g(A_2)^r} $ (the closures are in $A^r$) for every $r\geq \gen_K(A_1)$.
\end{prp}

\begin{proof}
	Fix some $r\in\N$. Let $W_{\leq m}$
	denote the set of tuples $a=(a_1,\dots,a_r)\in A^r$
	which generate a subalgebra of $A$ of dimension $m$ or less.
	
	We first claim that $W_{\leq m}$ is closed in $A^r$.
	Indeed, let $M$ denote the set of formal expressions obtained
	by repeatedly applying the operations of the multialgebra $A$
	on variables $x_1,\dots,x_r$. Every $(m+1)$-tuple
	$w=(w_1,\dots,w_{m+1})\in M^{m+1}$ determines a morphism $f_w:A^r\to \bigwedge_K^{m+1} A$
	given by $f_w(a)=w_1(a)\wedge\dots\wedge w_{m+1}(a)$.
	The subset $W_{\leq m}$ is the intersection $\bigcap_{w\in M^{m+1}} f_w^{-1}(0)$,
	so it is closed in $A^r$.
	
	Let $W_{m} = W_{\leq m}-W_{\leq (m-1)}$ (with $W_{-1}=\emptyset$).
	Then $W_m$ is open $W_{\leq m}$. We view $W_{\leq m}$ as a closed subvariety
	of $A^r$ and $W_m$ as an open subvariety of $W_{\leq m}$.
	
	Next, we claim that the function $\vphi: W_m\to X$ mapping an $r$-tuple
	$a\in W_m$ to the $K$-subalgebra that it generates is a morphism.
	To that end, we embed in $X$ in $\bbP(\bigwedge^m_K A)$ using the Pl\"ucker embedding.
	It is enough to show that the composition $W_m\to X\to \bbP(\bigwedge^m_K A)$,
	denoted $\vphi'$,
	is defined locally by rational functions.
	Let $a\in W_m$. Then there are $w_1,\dots,w_m\in M$ such
	that $w_1(a),\dots,w_m(a)$ form $K$-basis for the $K$-subalgebra generated
	by $a$. The image of $a$ under $\vphi'$
	is therefore the line in $\bigwedge^m_K A$ 
	generated by $w_1(a)\wedge\dots \wedge w_m(a)$. 
	The assignment $x\mapsto w_1(x)\wedge\dots\wedge w_m(x)$ defines $\vphi'$ in a 
	Zariski neighborhood of $a$ in $W_m$
	(namely, whenever it does not vanish),
	and by choosing bases to $A^r$ and $\bigwedge^m_K A$, one readily sees
	it is a polynomial function. Thus, $\vphi$ is a morphism.
	
	Now, for $i=1,2$, put
	$Y_i=\bigcup_{g\in G}g(A_i)^r$.
	Then $Y_i$ is a subset of $W_{\leq m}$
	and $U_i:=Y_i\cap W_m$ is the set of $r$-tuples
	in $A^r$ which generate a $K$-subalgebra of the form $g(A_i)$ with $g\in G$.
	In particular, $U_i$ is open in $Y_i$ and $\vphi(U_i)\subseteq Z_i$.
	When  $r\geq \gen_K(A_i)$, each $g(A_i)$ is generated by
	some $a\in U_i$, which implies  that $U_i$ is dense in $Y_i$
	and also that $\vphi(U_i)=Z_i$.
	
	For the sake of contradiction, suppose that
	$r\geq \gen_K(A_1)$, but   $\overline{Y_1}\subseteq \overline{Y_2}$.
	Then $U_1\subseteq \overline{Y_2}\cap W_m$.
	If $r<\gen_K(A_2)$, then we would have $Y_2\subseteq W_{\leq(m-1)}$,
	meaning that $\overline{Y_2}\subseteq W_{\leq (m-1)}$ and $\overline{Y_2}\cap W_m=\emptyset$,
	which contradicts with $U_1\neq\emptyset$.
	It must therefore be the case that $r\geq \gen_K(A_2)$,
	hence $U_1\subseteq\overline{Y_2}=\overline{U_2}$.
	Now,
	the continuity of $\vphi$ implies 
	that $Z_1=\vphi(U_1)\subseteq \vphi(\overline{U_2})\subseteq\overline{\vphi(U_2)}= \overline{Z_2}$, so $\overline{Z_1}\subseteq \overline{Z_2}$.
	This contradicts our assumption $\overline{Z_1}\nsubseteq \overline{Z_2}$,
	so we must have $\overline{Y_1}\nsubseteq \overline{Y_2}$.
\end{proof}

\section{The Case of Azumaya Algebras with Unitary Involution}
\label{sec:Az-alg-with-inv}

Recall our standing assumption that all algebras and algebras with involution are unital.

For the remainder of the paper, we focus on a specific   $K$-algebra  with involution.
Given $n\in\N$, let 
$A_n$ denote the unital $K$-algebra
$\nMat{K}{n}\times\nMat{K}{n} $
endowed with the involution  
\[(a,b)^*=(b^\trans, a^\trans).\]
Forms of $(A_n,*)$ are therefore   algebras with involution.

By the following proposition,
the forms of $(A_n,*)$ are precisely the degree-$n$
\emph{Azumaya algebras with unitary involution}.
Recall that  given a ring $R$, an $R$-algebra with an $R$-involution $(B,\tau)$
is said to be Azumaya with unitary involution if $B$ is an  Azumaya algebra over
its center $S$, $S$ is a quadratic \'etale $R$-algebra, and $\tau|_S$ coincides with the standard
$R$-involution of $S$ (equivalently, $\tau(s)=\Tr_{S/R}(s)-s$ for all $s\in S$). 
In this case, the rank of $B$ as an $S$-module is a perfect square $n^2$
and $n$ is called the degree of $B$ and denoted $\deg B$.
See
\cite[\S1.4]{First_2022_octagon} for further details.

\begin{prp}\label{PR:forms-of-A}
	Let $R$ be a $K$-ring and let $(B,\tau)$ be an $R$-algebra with an $R$-involution.
	Then $(B,\tau)$ is an $R$-form of $(A_n,*)$ if and only if $(B,\tau)$ is a degree-$n$ Azumaya
	algebra with a unitary involution over $R$. 
\end{prp}

\begin{proof}
	Suppose first that $(B,\tau)$ is a degree-$n$ Azumaya
	algebra with a unitary involution over $R$. 
	Observe that if $R'$ is a faithfully flat $R$-ring, then it is harmless to replace
	$(B,\tau)$ with $(B_{R'},\tau_{R'}):=(B\otimes_R {R'},\tau\otimes\id_{R'})$.
	By assumption, the center $Z=\Cent(B)$
	is quadratic \'etale over $R$ and $\tau|_Z$ is the standard $R$-involution of $Z$.
	Thus, $Z\otimes_R Z\cong Z\times Z$ as $Z$-rings
	\cite[III.4.1]{Knus_1991_quadratic_hermitian_forms}. We replace $(B,\tau)$ and $R$ with
	$(B_Z,\tau_Z)$ and $Z$. Now, $Z=R\times R$ and $\tau|_Z$ is the swap involution
	$(x,y)\mapsto (y,x)$. In this case, it is well-known 
	(\cite[Example~2.4]{First_2022_octagon}, for instance)	
	that 
	there is an $R$-algebra $B_0$ such that $B\cong B_0\times B_0^\op$,
	and under this isomorphism, $\tau$ is given by $\tau(x,y^\op)=(y,x^\op)$.
	We may therefore assume that $B=B_0\times B_0^\op$.
	Since $B$ is Azumaya of degree $n$ over $Z$, the $R$-algebra
	$B_0$ is Azumaya of degree $n$ over $R$ (because $B_0\cong B\otimes_Z R$ when
	we view $R$ as a $Z$-algebra via $(x,y)\mapsto x:Z\to R$). 
	Thus, there exists
	a faithfully flat $R$-ring $R'$ such that $(B_0)_{R'}\cong \nMat{R'}{n}$
	\cite[III.5.1.15]{Knus_1991_quadratic_hermitian_forms}. By replacing $(B,\tau)$
	with $(B_{R'},\tau_{R'})$, we   reduce  to the case
	$B = \nMat{R}{n}\times \nMat{R}{n}^\op$
	and $\tau(x,y^\op)=(y,x^\op)$.
	One now readily checks that $(x,y)\mapsto (x,(y^\trans)^\op)$
	is an isomorphism from $(A_{n,R},*_R)$ to $(B,\tau)$, so $(B,\tau)$ is an $R$-form of $(A_n,*)$.
	
	Conversely, suppose that $(B,\tau)$ is an $R$-form of $(A_n,*)$.
	Then there is a faithfully flat $R$-ring $R'$
	such that $(B_{R'},\tau_{R'})\cong (A_{n,R'},*_{R'})$.
	Since $A_{n,R'}$ is separable and projective over $R'$, the same holds
	for $B_{R'}$. This means that $B$ is separable and projective over $R$. (Indeed,
	$B$ is a finitely generated projective $R$-module
	by   \cite[III.2.1.3]{Knus_1991_quadratic_hermitian_forms},
	and it is separable over $R$ because the morphism $ \sum_i x_i\otimes y_i^\op\mapsto
	[b\mapsto x_iby_i]: B\otimes_R B^\op\to \End_R(B)$ becomes  an isomorphism
	after tensoring with $R'$, and is therefore an isomorphism.) 
	Consequently,
	$B$ is Azumaya over $Z:=\Cent(B)$. By \cite[Lem.~1.4]{First_2022_octagon},
	we have $Z_{R'}\cong \Cent(B_{R'})\cong R'\times R'$, so $Z$ is quadratic
	\'etale over $R$. Moreover, under the latter isomorphism, the restriction of $\tau_{R'}$
	to $Z_{R'}$ corresponds to the swap involution $(x,y)\mapsto (y,x)$, hence
	$\tau|_Z:Z\to Z$ is the standard $R$-involution of $Z$.
	It remains to check that the degree of $B$ (over $Z$)
	is $n$. By the isomorphism $B_{R'}\cong A_n\otimes_K R'$ we have   $\rank_{R} B=2n^2$,
	so $\rank_Z B=n^2$ and $\deg B=n$.
\end{proof}

\section{The Maximal Subalgebras of \texorpdfstring{$(A_n,*)$}{(An,*)}}
\label{sec:maximal-subalg}

Let $K$ and $(A_n,*)$ be as in Section~\ref{sec:Az-alg-with-inv} and suppose 
that $K$ is algebraically closed. Fixing $n$ throughout, we abbreviate $(A_n,*)$ to $A$.

Let $Z_r$ be the $K$-scheme
of $r$-tuples   $(a_1,\dots,a_r)\in A^r$ which do not generate $A$ (as a unital algebra with involution).
We would like to apply 
Proposition~\ref{PR:dim-Zr-detailed}  and Theorem~\ref{TH:dimZr-and-irred-comps}   
in order to find the dimension and irreducible components of $Z_r$.
To that end, we need to:
\begin{itemize}
    \item describe the  maximal (unital, $*$-invariant) subalgebras of $A$;
    \item describe (the $K$-points of) $G:=\Aut(A)$ and its identity connected component $G^\circ$;
    \item verify that the action of $G^\circ$ on the maximal subalgebras of $A$ has finitely many orbits with representatives $A_1,\dots,A_t$;
    \item compute the stabilizer $H_i$ of each   $A_i$ under the action of $G^\circ$;
    \item compute the number of generators of each   $A_i$.
\end{itemize}
The purpose of this section is to address these points. The conclusions for the structure of $Z_r$
are stated in the next sections.

\medskip

We set additional notation to be used throughout.

As the base field $K$ is fixed, we shall write $\M$ for $\M(K)$.
We write $\pi_i$ ($i=1,2$) for the  projection  from $A =\M\times \M$
to its  $i$-th factor.

Given a subspace $V\subseteq K^n$,
we let  $V^\perp=\{x\in K^n\suchthat \langle x,V\rangle=0\}$, where  $\Trings{\cdot,\cdot}$ is the standard  bilinear form $\langle x,y\rangle=x^\trans y$ on $K^n$.
We say that $V$ is nontrivial if $0\neq V\neq K^n$, and in this case,
define $S_V=\{a\in\M|aV\subseteq V\}$.

As usual,  $\GL_n$ is the algebraic group of invertible matrices in $\M$ (and also 
its $K$-points), and $\PGL $ is the quotient of $\GL_n$ by its center.
The image of $c\in \GL_n$ in $\PGL $ is denoted $[c]$.
The algebraic group of invertible matrices preserving the standard bilinear form
on $K^n$ is denoted $\Ob$; on the level of $K$-points,
$\Ob=\{a\in\GL_n|a^\trans a=I\}$.\footnote{
	The algebraic group 
	$\Ob$ coincides with the orthogonal group of a regular quadratic form
	if $\Char K\neq 2$, but this is not the case if $\Char K=2$.} 
When is $n$ even, the symplectic group of $n\times n$ matrices is $\Spn=\{a\in\GL_n|a^\trans\Omega a=\Omega\}$, where $\Omega=I_{n/2}\otimes\left(\begin{smallmatrix}0&-1\\1&0\end{smallmatrix}\right)$.
Given a nontrivial subspace $V\subseteq K^n$, we let
$S_V^\times=S_V\cap\GL_n$.
The images of $\Ob$, $\Spn$ and $S_V^\times$
under   the projection $\GL_n\rightarrow\PGL$ 
are denoted $\POb$,   $\PSp$ and $\PS_V^\times$,  respectively.

We say that $[c]\in\PGL$ is symmetric, resp.\ alternating,
if $c$ is; this is independent of the representative $c\in \GL_n$.
We also abbreviate $(c^\trans)^{-1}$ to $c^{-\trans}$.

\emph{Unless indicated otherwise, algebras are   algebras with involution,
and the terms ``subalgebra'' and ``generating set'' are to be understood in this context.}

\begin{prp}\label{PR:maximal-subalgebras}
The maximal subalgebras of $A$  are:
\begin{enumerate}[label=(\arabic*)]
    \item $A_V:=S_V\times S_{V^\perp}$, where $V$ is a nontrivial subspace of $K^n$, and
    \item $B_{[p]}:=\{(a,pap^{-1}) \where a\in\M\}$, where $[p]\in\PGL$ is symmetric or alternating.
\end{enumerate}
There are no repetitions in this list.
\end{prp}

\begin{proof}
Let us first show that $A_V$ and $B_{[p]}$ are indeed subalgebras of $A$.
This is clear if we forget the involution, so it remains to check that they are
invariant under $*$.

Let $V$ be a nontrivial subspace of $K^n$. 
To show that $A_V$ is closed under the involution, it suffices to show that $(S_V)^\trans=S_{V^\perp}$. 
Indeed, 
$a\in S_V$ 
$\iff$
$ax\in V$  for all $x\in V$  
$\iff$
$\langle x , a^\trans y \rangle= \langle ax, y\rangle =0$ for all $x\in V$ and $y\in V^\perp$
$\iff$
$a^\trans y\in V^\perp$  for all $y\in V^\perp$  
$\iff$ 
$a^\trans\in S_{V^\perp}$.

Let $p\in\GL_n$ be symmetric or alternating. The subspace $B_{[p]}$ 
is closed under $*$ because, given $a\in\M$, we may define $b=p^{-1}a^\trans p=p^{-\trans}a^\trans p^\trans$ 
(because $p\in\{\pm p^\trans\}$)
and get
$(a,pap^{-1})^*=(p^{-\trans}a^\trans p^\trans,a^\trans)= (b,pbp^{-1})\in B_{[p]}$.

Next, let us show that no two of the subalgebras listed in (1) and (2) contain each other. If $V$, $W$ are distinct nontrivial subspaces of $K^n$, then elementary linear
algebra implies $S_V \nsubseteq S_W $ and so $A_V \nsubseteq A_W$. If $p,q\in\GL_n$ are 
symmetric or alternating, then $B_{[p]}\subseteq B_{[q]}$ if and only if $pap^{-1}=qaq^{-1}$ for all $a\in\M$  if only if $[p]=[q]$. If $V$ is a nontrivial subspace of $K^n$ and $p\in\GL_n$ is symmetric or alternating, then $A_V \nsubseteq B_{[p]}$ because $(I,0)\in A_V-B_{[p]}$, and $B_{[p]} \nsubseteq A_V$ because $\pi_1(B_{[p]})=\M\nsubseteq S_V=\pi_1(A_V)$.

To finish, we show that for any proper subalgebra $B\subseteq A$, there is a nontrivial subspace $V\subseteq K^n$ such that $B\subseteq A_V$, or a symmetric or alternating matrix $p\in\GL_n$ such that $B\subseteq B_{[p]}$.

Suppose first that $\pi_1(B)\neq \M$ or $\pi_2(B)\neq \M$. Since
$\pi_2(B)=\pi_2(B^*)=\pi_1(B)^\trans$, we have $\pi_1(B)\subsetneq \M$ in both cases. 
Then $\pi_1(B)$ is a proper    subalgebra of $\M$ (considered as a unital algebra without involution).
By Burnside's theorem (see \cite[Lem.~7.3]{Lam_1991_first_course}, for instance), 
there is a nontrivial subspace $V\subseteq K^n$ such that $\pi_1(B)\subseteq S_V$ and thus $B\subseteq S_V\times\M$. We have shown that $(S_V)^\trans=S_{V^\perp}$ and thus $B=B^*\subseteq \M^\trans\times (S_V)^\trans=\M\times S_{V^\perp}$. As a result, $B\subseteq (S_V\times\M)\cap (\M\times S_{V^\perp})= A_V$.

Suppose now that $\pi_1(B)=\pi_2(B)=\M$. We claim that in this case $\pi_2:B\rightarrow\M$ is an isomorphism of algebras without involution. For the sake of contradiction, suppose it is not. Since $\pi_2$ is surjective, its kernel is not trivial, i.e., there is $0\neq b\in\M$ such that $(b,0)\in B$. Let $a\in\M$. Since $a\in\pi_1(B)$, there is $c\in\M$ such that $(a,c)\in B$. Therefore, $(ab,0)=(a,c)(b,0)\in B$ and $(ba,0)=(b,0)(a,c)\in B$, which means that $\langle b\rangle\times0\subseteq B$, where $\langle b\rangle$ is the two-sided ideal of $\M$ generated by $b$. The ring $\M$ is simple and $b\neq 0$, so $\langle b\rangle=\M$  and    $\M\times 0\subseteq B$.
As a result, $0\times\M=(\M\times0)^*\subseteq B$, and it follows that $B=A$,  a contradiction. Therefore, $\pi_2$ is an isomorphism. This implies that $\dim B =n^2$, so $\pi_1:B\rightarrow\M$ is also an isomorphism.

Define $f=\pi_2\circ\pi_1^{-1}$. 
Then $f$ is an automorphism of $\M$ and
$B=\{(a,f(a))|a\in\M\}$. By the Skolem--Noether theorem, there exists $p\in\GL_n$ such that $f(a)=pap^{-1}$ for
all $a\in \M$. All we have left to prove is that $p$ is symmetric or alternating. Let $a\in\M$. We have $(a^\trans,pa^\trans p^{-1})\in B$ and thus also $(p^{-\trans}ap^\trans,a)=(a^\trans,pa^\trans p^{-1})^*\in B$, hence $a=p(p^{-\trans}ap^\trans)p^{-1}$, or rather, $p^\trans p^{-1}a=ap^\trans p^{-1}$. 
As this holds for every $a\in\M$, the matrix $p^\trans p^{-1}$ is scalar, i.e., there exists $\varepsilon\in K$ such that $p^\trans p^{-1}=\varepsilon I$, or equivalently, $p^\trans=\varepsilon p$. We have $p=p^{\trans\trans}=(\veps p)^\trans=\veps^2 p$,
so $\veps\in \{\pm1\}$.
If $\veps=1$ then $p$ is symmetric, and otherwise $\veps=-1$ and $\Char K\neq 2$, so $p$ is alternating.
\end{proof}

We proceed with describing $G=\Aut_K(A)$ and $G^\circ$.
We denote the nontrivial element of the permutation group $S_2$ by $\sigma$.

\begin{prp}\label{PR:aut-group}
The action of $\PGL$ on $A$ given by $[c](a,b)=(cac^{-1},c^{-\trans}bc^\trans)$ and the action of $S_2$ on $A$ given by $\sigma(a,b)=(b,a)$ give rise to isomorphisms $G\cong \PGL\rtimes S_2$ and $G^\circ\cong\PGL$.
Here, $S_2$ acts on $\PGL$ by $\sigma([c])=[c^{-\trans}]$.
\end{prp}

\begin{proof}
It is straightforward to check that the action of $\PGL$ is well-defined
and that $\PGL$ and $S_2$ act on $A$ by algebra automorphisms.
Moreover, one readily checks that $\sigma [c]\sigma^{-1}(a,b)=[c^{-\trans}](a,b)$ for all $(a,b)\in A$ and $c\in \GL_n$,
so the induced homomorphisms $\PGL\to G$, $S_2\to G$ glue into a homomorphism $\PGL\rtimes S_2\to G$,
which is clearly injective.
We need to show that it is also surjective.

Let $f\in G$. The center of $A$ is $K \times K $, so the only nontrivial central idempotents in $A$ are $(I,0)$ and $(0,I)$. Thus, $f$ either fixes both of them or exchanges them.

If $f$ fixes $(I,0)$ and $(0,I)$, then $f$ is a $(K\times K)$-algebra automorphism
of $A=\M \times \M$. This means that $f(a,b)=(g(a),h(b))$ for some   $g,h\in\Aut_K(\M)$, so by the Skolem--Noether theorem, there
are $c,d\in\GL_n$ such that $f(a,b)=(cac^{-1},dbd^{-1})$. 
The automorphism $f$ respects the involution $*$, so for every $a,b\in \M $,
we have 
\[c^{-\trans}a^\trans c^\trans=\pi_2((f(a,b))^*)=\pi_2(f((a,b)^*))=da^\trans d^{-1}.\] 
This means that $[d]=[c^{-\trans}]$ and so $f(a,b)=(cac^{-1},c^{-\trans}bc^\trans)=[c](a,b)$.

If $f(I,0)=(0,I)$, then $f \circ \sigma$   is an automorphism of $A$ that fixes $(I,0)$ and $(0,I)$. By what 
we have shown, there is $[c]\in \PGL$ such that $f\circ \sigma = [c]$, so
$f=[c]\sigma$. This shows that $G\cong \PGL\rtimes S_2$.

Since $\PGL$ is connected, the identity connected component of $\PGL\rtimes S_2$ is $\PGL$
(more precisely, its copy in  $\PGL\rtimes S_2$).
The isomorphism $\PGL\rtimes S_2\cong G$ therefore restricts to an isomorphism $\PGL\cong G^\circ$.
\end{proof}

Henceforth, we shall freely identify $G^\circ$ with $\PGL$
as in Proposition~\ref{PR:aut-group}.

\begin{lem}\label{LM:action}
	Let
	$[c]\in \PGL=G^\circ$, let  
$V$ be a nontrivial subspace of $K^n$ and  let $p\in \GL_n$ be symmetric or alternating.	
	Then, with notation as in Proposition~\ref{PR:maximal-subalgebras},
we have $[c]A_V=A_{cV}$ and  $[c]B_{[p]}=B_{[c^{-\trans}pc^{-1}]}$.
\end{lem}

\begin{proof}
If $a\in S_V$, then $(cac^{-1})cV=caV\subseteq cV$ and so $cac^{-1}\in S_{cV}$. 
This means that $\pi_1([c]A_V)= \{cac^{-1}\where a\in S_V\}\subseteq S_{cV}$. 
The only maximal subalgebra $B$ of $A$ that satisfies $\pi_1(B)\subseteq S_{cV}$ is $A_{cV}$, so 
we must have $[c]A_V=A_{cV}$. This proves the first assertion.
Next, $[c]B_{[p]}=\{(cac^{-1},c^{-\trans}pap^{-1}c^\trans)\,|\,a\in\M\}$, and by substituting $a$ with $c^{-1}ac$, we get $[c]B_{[p]}=\{(a,c^{-\trans}pc^{-1}acp^{-1}c^\trans)\,|\,a\in\M\}=B_{[c^{-\trans}pc^{-1}]}$.
\end{proof}

\begin{prp}\label{PR:representatives}
The action of $G^\circ$ on the maximal subalgebras of $A$ has finitely many orbits, and they are represented by the following subalgebras:
\begin{enumerate}[label=(\arabic*)]
    \item $A_{V(i)}$ for  $i\in\{1,\dots,n-1\}$, where $V(i)=  Ke_1+ \dots +K e_i$;
    \item $B_{[I]}$ and, if $n$ is even, also $B_{[\Omega]}$.
\end{enumerate}
Here, $e_1,\dots,e_n$ is the standard basis of $K^n$.
\end{prp}

\begin{proof}
Let $V\subseteq K^n$ be a nontrivial subspace of dimension $i$. 
By Lemma~\ref{LM:action}, the $G^\circ$-orbit of $A_V$ is $\{A_{cV}\,|\, c\in \GL_n\}=\{A_W\,:\,\dim W=i\}$.
This means that $A_V$ is in the $G^\circ$-orbit of $A_{V(i)}$ and not in any other $G^\circ$-orbit.

Let $p\in\GL_n$ be symmetric or alternating.
Since $[c]B_{[p]}=B_{[c^{-\trans}pc^{-1}]}$, the $G^\circ$-orbit of $B_{[p]}$
consists of those $B_{[q]}$ with $q$ congruent to $p$. 
We noted in Section~\ref{sec:conventions} that $p\sim I$
if $p$ is  symmetric and not  alternating, $p\sim\Omega$ if $p$ is alternating,
and $I\nsim \Omega$.
Thus, $B_{[p]}$ is  in the $G^\circ$-orbit of
$B_{[I]}$, or that of $B_{[\Omega]}$ (provided $n$ is even), and these two orbits are different.
\end{proof}

\begin{prp}\label{PR:stabilizers}
The stabilizers of the subalgebras from Proposition~\ref{PR:representatives}
under the action of $G^\circ=\PGL$ are given as follows:
 \begin{enumerate}[label=(\arabic*)]
    \item $\Stab_{G^\circ}(A_{V(i)})= \PS_{V(i)}^\times$ for $1\le i\le n-1$;
    \item $\Stab_{G^\circ}(B_{[I]})=\POb$ and, if $n$ is even, $\Stab_{G^\circ}(B_{[\Omega]})=\PSp$.
\end{enumerate}
\end{prp}

\begin{proof}
Let $1\le i\le n-1$. 
By Lemma~\ref{LM:action},
$\Stab(A_{V(i)})=\{[c]\in\PGL \suchthat cV(i) =V(i)\}=\PS_{V(i)}^\times$,
so (1) holds.

To prove (2), it suffices to show that if $p\in\{I,\Omega\}$, then $\Stab(B_{[p]})=\{[c]\where c\in \GL_n,\, c^\trans pc=p\}$. By Lemma~\ref{LM:action},
$\Stab(B_{[p]})=
\{[c]\in\PGL\suchthat [c^{-\trans}pc^{-1}]=[p]\}
=\{[c]\in\PGL\suchthat [c^\trans pc]=[p]\}$.
It is now clear that $\{[c]\where c\in \GL_n,\, c^\trans pc=p\}\subseteq \Stab(B_{[p]})$. The converse  holds because if $[c^\trans pc]=[p]$, then $\lambda c^\trans pc=p$ for some $\lambda\in K^\times$, so upon setting $d=\sqrt{\lambda}c$, we have $[c]=[d]$ and $d^\trans pd=p$.
\end{proof}

The remainder of this section is dedicated to determining the number of generators for the algebras considered
in Proposition~\ref{PR:representatives}.

\begin{lem}\label{LM:identities}
	Let $K$ be any field, $n\in\N$ and $i,j\in\{1,\dots,n\}$.
	As usual,  let $e_{i,j}\in\nMat{K}{n}$ be the matrix with $1$ in the $(i,j)$-entry and $0$ elsewhere. Define $u=\sum_{i=1}^{n-1} e_{i,i+1}$ and let $d_{i,j}$ be $e_{i,i}$ if $i=j$ and $e_{i,j}+e_{j,i}$ otherwise.  Then:
	\begin{enumerate}[label=(\roman*)]
		\item $u^ke_{i,j}u^\ell= \left\{\begin{array}{ll}
e_{i-k,j+\ell} & \text{$k<i$ and $\ell\leq n-j$} \\
0 &  \text{otherwise}
\end{array}
\right.$ for all integers $k,\ell\geq 0$.
		\item $d_{i,j}^\ell=d_{i,j}$ for every odd $\ell\in\N$.
		\item $d_{i,j}u^{j-i}d_{i,j}u^{j-i}d_{i,j}=e_{j,i}$, provided $i<j$.
	\end{enumerate}
\end{lem}

\begin{proof}
	(i)  Note that $ue_{i,j}$ is $e_{i-1,j}$ if $1<i$ and $0$ otherwise, and $e_{i,j}u$ is $e_{i,j+1}$ if $j<n$ and $0$ otherwise. The identity now follows by induction.
	
	(ii) This holds because $d_{i,j}^3=d_{i,j}$, which is seen by computation.
	
	(iii) Suppose first that $j-i<i$ and $j\le n-(j-i)$.
	By (i), 
	\begin{align*}
	d_{i,j}u^{j-i}d_{i,j}u^{j-i}d_{i,j}
	&=
	d_{i,j}u^{j-i}(e_{j,i}+e_{i,j})u^{j-i}d_{i,j}
	=
	d_{i,j}(e_{i,j}+e_{2i-j,2j-i})d_{i,j}
	\\
	&=
	d_{i,j}e_{i,j}d_{i,j}+d_{i,j}e_{2i-j,2j-i}d_{i,j}
	=
	e_{j,i}.
	\end{align*}
	The other case is handled similarly, except the factor $e_{2i-j,2j-i}$ does not occur.
\end{proof}

We will only need a special case of the following result, but we prove it in full generality
because it may be of general interest.

\begin{prp}\label{PR:central-simple}
Let $K$ be any field (possibly finite) and let $(B,\tau)$ be a degree-$n$ central simple algebra with involution over $K$.\footnote{
	Our definition of a central simple algebra with involution is as in \cite{Knus_1998_book_of_involutions}. Thus, in the unitary case, the center of $B$ is 
	a quadratic \'etale $K$-algebra and the restriction of $\tau$ to the center is its
	unique nontrivial $K$-automorphism.
} Then $\gen_K(B,\tau)=1$, unless $n=2$ and the involution $\tau$ is symplectic, in which case $\gen_K(B,\tau)=2$, or $n=1$ and $\tau$ is orthogonal,  where $\gen_K(B,\tau)=0$.
\end{prp}

\begin{proof}

Contrary to our convention elsewhere in this section, we will   consider
algebras without involution. When discussing an algebra with involution, the involution
will be written  explicitly, e.g., $(B,\tau)$. 
We will use the matrices introduced in Lemma~\ref{LM:identities}.

We   first prove that $\gen_K(B,\tau)\leq 1$ in three special cases.

\smallskip

{\it Case I. $B=\M$ and $\tau(a)=d^{-1} a^\trans d$ 
with $d=\mathrm{diag}(\lambda_1,\dots,\lambda_n)\in\GL_n$.}
By Lemma~\ref{LM:identities}(i), $e_{i,j}=u^{n-i}e_{n,1}u^{j-1}$ for
all $i,j\in\{1,\dots,n\}$, and so $u$ and $e_{n,1}$ generate $\M$. Note that $u^{n-1}=e_{1,n}$ and therefore $\frac{\lambda_n}{\lambda_1}\tau(u^{n-1})=e_{n,1}$. 
This shows that $u$ generates $(B,\tau)$.

\smallskip

{\it Case II. $B=\M$, $\tau(a)=\Omega a^\trans\Omega^{-1}$ and $n>2$ is even.}
Again, we will show that $u$ generates $(B,\tau)$. 
Indeed, using Lemma~\ref{LM:identities}(i),
\begin{align*}
\tau(u^{n-3}\tau(u^{n-1})u^{n-3})
&=
\tau(u^{n-3}\tau(e_{1,n})u^{n-3})
=
\tau(u^{n-3}(-e_{n-1,2})u^{n-3})
\\
&=\tau(-e_{2,n-1})=e_{n,1}.
\end{align*}
We already observed that $u$ and $e_{n,1}$ generate $\M$, so this case is done.

\smallskip

{\it Case III. $(B,\tau)=(A_n,*)$.}
We claim that $(u,d_{1,n})$ generates $(B,\tau)$. Let $\ell \ge n$ be odd. Since $u^\ell=0$, we have $(u,d_{1,n})^\ell=(0,d_{1,n})$ (Lemma~\ref{LM:identities}(ii)). 
This means that both $(u,0)$ and $(0,d_{1,n})^*=(d_{1,n},0)$ are in the subalgebra of $(A_n,*)$ generated by $(u,d_{1,n})$. By Lemma~\ref{LM:identities}(iii),  $(e_{n,1},0)$ is also in that subalgebra, so   it contains $\M\times0$. It must also contain $0\times\M=(\M\times0)^*$,
and we conclude that $(u,d_{1,n})$ generates $(A_n,*)$. 

\smallskip

Returning to the general case, we now  show that $\gen (B,\tau)\leq 1$ unless
$n=2$ and $\tau$ is symplectic.
Since $\tau$ is orthogonal, symplectic, or unitary, $B$ is a  $K$-form of the algebras with involution
considered in Cases I, II and III, respectively (the unitary case is a consequence of Proposition~\ref{PR:forms-of-A}). If $K$ is infinite, then this implies the proposition by \cite[Prop.~4.1]{First_2017_number_of_generators}. 

Suppose now that $K$ is finite. 
By Wedderburn's theorem, every central simple $K$-algebra is isomorphic to a matrix algebra over $K$.
Thus, if $\tau$ is of the first kind, $B\cong \nMat{K}{n}$ and otherwise,
either $B\cong \nMat{K}{n}\times\nMat{K}{n}$, or $B\cong \nMat{L}{n}$, where $L$ is the unique quadratic extension of $K$.   
If $\tau$ is orthogonal, then it is the adjoint involution of a symmetric diagonalizable bilinear form \cite[\S2]{Knus_1998_book_of_involutions}. This case is therefore covered (up to isomorphism) by Case I (the involution in this case is adjoint to the diagonal bilinear form $b(x,y)=\sum_{i=1}^n\lambda_i x_iy_i$ on $K^n$).
Similarly, the case where $\tau$ is symplectic and $n>2$ is covered by Case II.
In the case where $\tau$ is unitary and $B\cong \nMat{K}{n}\times\nMat{K}{n}$,
we have $(B,\tau)\cong (A_n,*)$ (e.g., see \cite[Prp.~2.14]{Knus_1998_book_of_involutions}), 
so we are done by Case III.
It remains to consider the case where $\tau$ is unitary and $B\cong \nMat{L}{n}$.
In this case,   $(B,\tau)\cong(\nMat{L}{n},\sigma)$,
where $\sigma$ acts on a matrix   by transposing   and then applying the   nontrivial automorphism of $L/K$ to each entry (see \cite[Chp.~10, Example~1.6(i)]{Scharlau_1985_quadratic_and_hermitian_forms} and also p.~302 in this source). Let $\alpha\in L-K$. We claim that $\alpha I+u$ generates $(\nMat{L}{n},\sigma)$. Indeed, let $m$ be a power of $|L|$ greater than $n$. 
Since $\alpha I$ and $u$ commute and $m$ is a power of $|L|$, and in particular of $\Char L$, 
we have $(\alpha I+u)^m=\alpha^m I+u^m=\alpha I$. Thus, $\alpha I$ and $u$ are in the $K$-subalgebra of $\nMat{L}{n}$ generated by $\alpha I+u$. As $u$ generates $(\nMat{K}{n},\trans)$ (Case~I) and $\alpha$ generates $L$, the element $\alpha I+u$ generates $(\nMat{L}{n},\sigma)$. This completes
the proof that $\gen_K(B,\tau)\leq 1$.

Notice that the (unital, involutive) subalgebra of $(B,\tau)$
generated by the empty set is $K$,
so if $B\neq K$, we must have $\gen_K(B)>0$. Together with what we showed, this proves
the proposition in all cases except when $n=2$ and $\tau$ is symplectic.

To finish, suppose $n=2$ and $\tau$ is symplectic.
Then for all $b\in B$, we have
$b+\tau(b)\in K$  (\cite[Prop.~2.6]{Knus_1998_book_of_involutions})
and $b^2=\Trd_{B/K}(b)b-\Nrd_{B/K}(b)$ (by the Cayley--Hamilton theorem),
hence the subalgebra of $(B,\tau)$ generated by $b$ is contained in $K+Kb$.
This means that $\gen_K(B,\tau)>1$. On the other hand,  
$\gen_K(B,\tau)\leq \gen_K(B)\leq 2$ because $B$ is a central simple
$K$-algebra.
\end{proof}

\begin{prp}\label{PR:subalgebra-generators}
Let $B$ be a maximal subalgebra of $A$. 
Then $\gen(B)\leq 1$ unless $n=2$ and $B=B_{[\Omega]}$ (notation
as in Proposition~\ref{PR:representatives}), in which case
$\gen(B_{[\Omega]})=2$.
\end{prp}

\begin{proof}
By replacing $B$ with $[c]B$ for some $[c]\in\PGL$, we may assume $B$
is one of the algebras listed in Proposition~\ref{PR:representatives}.
(Note   that when $n=2$ and $B=B_{[\Omega]}$, Lemma~\ref{LM:action}
tells us that  $[c]B=B$, because $[d^{\trans} \Omega_2 d]=[\Omega_2]$
for every $d\in\GL_2$.)
We shall   use the matrices defined in  Lemma~\ref{LM:identities}.

Suppose first that $B=A_{V(k)}$, where $1\le k \le n-1$. Let $\alpha\in K-\{0,\pm1\}$ and denote $g=(u,d_{1,k}+\alpha d_{k+1,n})\in A_{V(k)}$. 
Let $\langle g\rangle$ be the (unital, involutive)
subalgebra of $A_{V(k)}$ generated by $g$. We will show that $\langle g\rangle=A$.

Let $\ell\ge n$ be odd. Since $u^n=0$ and $d_{1,k}d_{k+1,n}=d_{k+1,n}d_{1,k}=0$, 
we have $g^\ell=(0,d_{1,k}+\alpha^\ell d_{k+1,n})=(0,d_{1,k})+\alpha^\ell(0,d_{k+1,n})$ and similarly $g^{\ell+2}=(0,d_{1,k})+\alpha^{\ell+2}(0,d_{k+1,n})$. 
Since $\alpha\notin\{0,\pm1\}$ we have $\alpha^\ell\ne\alpha^{\ell+2}$ and so the coefficient vectors $(1,\alpha^\ell)$ and $(1,\alpha^{\ell+2})$ are linearly independent. Hence, $(0,d_{1,k})$ and $(0,d_{k+1,n})$ are linear combinations of $g^\ell$ and $g^{\ell+2}$ and so $(0,d_{1,k}),(0,d_{k+1,n})\in\langle g\rangle$. 
Therefore, $(d_{1,k},0)=(0,d_{1,k})^*$ and $(d_{k+1,n},0)=(0,d_{k+1,n})^*$ are also in $\langle g\rangle$. 
By Lemma~\ref{LM:identities}(iii), this implies   $(e_{k,1},0),(e_{n,k+1},0)\in \langle g\rangle$. 
Now, by Lemma~\ref{LM:identities}(i), if $1\leq i\le k$, then $(e_{i,j},0)=(u^{k-i}e_{k,1}u^{j-1},0)=g^{k-i}(e_{k,1},0)g^{j-1}\in\langle g\rangle$, 
and if $k<j\leq n$, then $(e_{i,j},0)=(u^{n-i}e_{n,k+1}u^{j-k-1},0)=g^{n-i}(e_{k,1},0)g^{j-k-1}\in\langle g\rangle$.
The matrices $e_{i,j}$ with $i\le k$ or $ k<j$ form a basis of $S_{V(k)}$, so $S_{V(k)}\times0\subseteq\langle g\rangle$. Thus, $0\times S_{V(k)^\perp}=(S_{V(k)}\times0)^*\subseteq\langle g\rangle$ and we conclude that $\langle g\rangle=A_{V(k)}$.

Suppose next that $B=B_{[I]}$.
Then $\pi_1:B\to (\M,\trans)$ is an isomorphism of algebras with involution,
hence $\gen(B_{[I]})\leq 1$ by Proposition~\ref{PR:central-simple}.

Finally, if $B=B_{[\Omega]}$, then $\pi_1$ induces an isomorphism from $B$ to
$\M$ with the symplectic involution $a\mapsto \Omega a^\trans \Omega^{-1}$,
and again we finish by Proposition~\ref{PR:central-simple}.
\end{proof}

\section{The Dimension of \texorpdfstring{$Z_r$}{Zr} in the Case \texorpdfstring{$A=(A_n,*)$}{A=(An,*)}}
\label{sec:dim-Zr}

In this section, we apply Theorem~\ref{TH:dimZr-and-irred-comps} and the
results of Section~\ref{sec:maximal-subalg} in order to prove 
Theorem~\ref{TH:codim-unitary}, i.e.,
finding $\dim Z_r$ for the algebra $(A_n,*)$ of Section~\ref{sec:Az-alg-with-inv}.
We then prove Corollary~\ref{CR:generators-of-Azumaya-alg} and a few more
results about the number of generators of Azumaya algebras with involution.

Throughout, $K$ is field,    $A=(A_n,*)$ is as in Section~\ref{sec:Az-alg-with-inv}
and   $Z_r$ is the scheme of $r$-tuples not generating $A$
(Section~\ref{sec:non-generators}).


\begin{proof}[Proof of Theorem~\ref{TH:codim-unitary}]

Recall that we need to show that $\dim Z_r = 2rn^2 - (2r-1)(n-1)$ if $n>1$ and $\dim Z_r=r$
otherwise.
It is enough to show this when   $K$ is algebraically closed.
To that end, we apply Theorem~\ref{TH:dimZr-and-irred-comps}(ii).

Let $G$ being the identity
connected component of $\Aut_K(A)$, which we identify with $\PGL$
using  Proposition~\ref{PR:aut-group}. 
The  action of $G$ on the maximal subalgebras of $A$ has finitely many orbits
with representatives listed in Proposition~\ref{PR:representatives}.
Using the notation of that proposition, let
$s(i,r)=r\dim A_{V(i)}-\dim\Stab_G(A_{V(i)})$,
$s'(r)=r\dim B_{[I]}-\dim \Stab_G(B_{[I]})$
and for even $n$,
let $s''(r)=  r\dim B_{[\Omega]}-\dim \Stab_G(B_{[\Omega]})$.

Suppose $n>1$. By Proposition~\ref{PR:stabilizers}, we have
\begin{align*}
s(i,r)&=r\dim A_{V(i)}-\dim\Stab_G(A_{V(i)}) =
2r\dim S_{V(i)} - \dim \PS_{V(i)}^\times \\
&=2r[ n^2-i(n-i)] - [n^2 -i(n-i) -1] =(2r-1)(n^2-i(n-i))+1.
\end{align*}
Similarly, we have
\begin{align*}
s'(r)& =r\dim B_{[I]} - \dim \Stab_G(B_{[I]})=rn^2-\dim \POb, \\
s''(r)& =r\dim B_{[\Omega]} - \dim \Stab_G(B_{[\Omega]})=rn^2-\dim \PSp,
\end{align*}
where in the second equation, $n$ is assumed to be even.
It is well-known that $\dim \POb=\frac{n(n-1)}{2}$ when
$\Char K\neq2$ and $\dim \PSp=\frac{n(n+1)}{2}$ in any characteristic. 
As for $\dim \POb$ in the case $\Char K=2$, since $I=-I$,
on the level of $K$-points, the quotient map $\Ob\to \POb$
is a homeomorphism, so $\dim \Ob=\dim \POb$.
The algebraic group  $\Ob$ is the scheme of $n\times n$ matrices $a$ satisfying
$a^\trans a=I$, so it is the zero locus of $\frac{n(n+1)}{2}$ equations in $n^2$ variables,
which means that $\dim \Ob\geq n^2 - \frac{n(n+1)}{2}=\frac{n(n-1)}{2}$.
Putting everything together gives 
\[
s'(r)\leq rn^2-\frac{n(n-1)}{2} \qquad\text{and}\qquad
s''(r)= rn^2-\frac{n(n+1)}{2} .
\]
The maximum of $s(i,r)$ with $i\in \{1,\dots,n-1\}$ is attained at $i=1$ and $i=n-1$,
and elementary analysis implies that
\[
s(1,r)=(2r-1)(n^2-n+1)+1=rn^2+(r-1)(n-1)^2 - n+(r+1)
\]
is greater than $s'(r)$ and $s''(r)$.
By Proposition~\ref{PR:subalgebra-generators}, $\gen(A_{V(1)})=1\leq r$,
so   
Theorem~\ref{TH:dimZr-and-irred-comps}(ii) tells us that
\begin{align*}
\dim Z_r & =\dim \PGL + s(1,r)
=(n^2-1)+ (2r-1)(n^2-n+1)+1
\\
&= 2rn^2-(2r-1)(n-1).
\end{align*}

In the remaining case $n=1$, $A$ has only one maximal subalgebra, 
namely $B_{[I]}\cong (K,\id_K)$, it can be generated by $0$ elements
(using the unity),
and $\dim G=0$.
Thus, by Theorem~\ref{TH:dimZr-and-irred-comps}(ii), $\dim Z_r = r\dim B_{[I]} = r$.
\end{proof}


\begin{proof}[Proof of Corollary~\ref{CR:generators-of-Azumaya-alg}]
	Let $R$ be a ring of Krull dimension $d$ finitely generated over an infinite field
	$K$ and let $(B,\tau)$ be an Azumaya algebra with unitary involution of degree $n>1$ over $R$. 
	Then $(B,\tau)$
	is an $R$-form of $A=(A_n,*)$ (Proposition~\ref{PR:forms-of-A}). By
	Theorem~\ref{TH:codim-unitary},
	we have $c_A(r):=\dim A^r-\dim Z_r =(2r-1)(n-1)$. 
	One readily checks that $c_A(r)>d$ if and only if
	$r>\frac{d}{2n-2}+\frac{1}{2}$, or equivalently,
	if $r\geq \floor{\frac{d}{2n-2}+\frac{3}{2}}$. The corollary
	is now a consequence of Theorem~\ref{TH:cAr-and-number-of-generators}.
\end{proof}

\begin{remark} 
In Corollary~\ref{CR:generators-of-Azumaya-alg}, we view
$(B,\tau)$   as a \emph{unital} algebra with involution.
This means in particular that the unity may be used in generating 
the algebra by a subset.
However, by  \cite[Lem.~2.2]{First_2022_generators_of_alg_over_comm_ring} (which extends verbatim to algebras with involution), the number of generators of a form of $(A_n,*)$  
is not changed if we consider it as a non-unital algebra with involution.
\end{remark}

\begin{remark}\label{RM:lower-bounds}
	Let $R$ be a ring and let $B$ be an Azumaya algebra over $R$.
	Then $B^+:=B\times B^\op$ with the involution $\sigma(x,y^\op)=(y,x^\op)$
	is an Azumaya algebra with a unitary involution over $R$.
	It is straightforward to see that if $(B^+,\sigma)$ can be generated
	by $r$ elements, then $B$ can be generated by $2r$ elements, hence
	\[
	\gen_R( B^+,\sigma)\geq  \Ceil{\frac{1}{2}\gen_R(B)} .
	\]
	
	Suppose $\Char K=0$. In \cite[Thm.~1.5(b)]{First_2022_generators_of_alg_over_comm_ring}
	and \cite[Thm.~1.3]{Gant_2024_space_gens_2by2_matrix}, it is shown
	that for every $d\geq 0$ and $n\geq 2$, there is a smooth $K$-ring  $R$ of Krull dimension $d$
	carrying
	an Azumaya algebra $B$ of degree $n$ such that
	$\gen_R(B)\geq \Floor{\frac{d}{2n-2}}+2$ if $n>2$ and $\gen_R(B)\geq 2\Floor{\frac{d}{4}}+2$
	if $n=2$. Since
	$\Ceil{\frac{1}{2}\Floor{x}}=\Floor{\frac{x+1}{2}}$,	
	for that $R$-algebra $B$, we get
	\[
	\gen_R(B^+,\sigma)\geq \left\{\begin{array}{ll}
	\Floor{\frac{d }{4n-4} +\frac{3}{2}}  & n>2 \\[0.2cm]
	\Floor{\frac{d}{4}}+1 & n=2,
	\end{array}\right.
	\]
	which is  about half of the upper bound given in Corollary~\ref{CR:generators-of-Azumaya-alg}.
	
	By \cite[Thm.~1.5(a)]{First_2022_generators_of_alg_over_comm_ring}, every degree-$n$ Azumaya algebra over
	a finitely generated $K$-ring $R$ of Krull dimension $d$ can be generated by $\Floor{\frac{d}{n-1}}+2$
	elements. If this upper bound is tight,
	i.e., there are examples $B$ meeting this bound, then we would get
	$\gen_R(B^+,\sigma)\geq \Ceil{\frac{1}{2}\Floor{\frac{d}{n-1}}}+1=\Floor{\frac{d}{2n-2}+\frac{3}{2}}$,
	meaning that the upper bound of Corollary~\ref{CR:generators-of-Azumaya-alg} is  tight as well.
\end{remark}

Note that the upper bound of Corollary~\ref{CR:generators-of-Azumaya-alg}
applies only when the base ring $R$ is finitely generated over an infinite field.
By 
building on \cite[Thm.~1.2]{First_2017_number_of_generators}
and our Proposition~\ref{PR:central-simple},
we can   give a weaker
upper bound 
which  
applies 
over any noetherian ring $R$
and to Azumaya algebras with involution of any type.

Here, following \cite[\S1.4]{First_2022_octagon}, we call an $R$-algebra with $R$-involution
$(B,\tau)$ Azumaya if $B$ is Azumaya over its center $Z$, $Z$ is finite \'etale over
$R$ and the structure map $R\to \{b\in Z\suchthat \tau(b)=b\}$ is an isomorphism.
Given a maximal ideal $\frakm$ of $R$, we write $B(\frakm)=B/B\frakm$
and $\tau(\frakm)$ for the involution of $B(\frakm)$ induced by $\tau$. The pair
$(B(\frakm),\tau(\frakm))$ is then a central simple algebra with involution over $R/\frakm$.\footnote{
	This fact is explained in \cite[\S1.4]{First_2022_octagon} under the assumption $2\in R^\times$, but it holds
	without it. Indeed,  $B(\frakm)$ is Azumaya over its center, which may be 
	identified with $Z(\frakm)$ by \cite[Lem.~1.4]{First_2022_octagon}.
	The $R$-involution $\tau|_Z:Z\to Z$ is \emph{standard} in the sense that
	for all $z\in Z$, we have $\tau(z)+z,\tau(z)z\in R$, so this also holds for 
	the $R/\frakm$-involution $\tau(\frakm):Z(\frakm)\to Z(\frakm)$.
	Using the fact that $Z(\frakm)$ is \'etale of dimension $1$ or $2$ over $R/\frakm$,
	it follows easily that  the fixed subring of $\tau(\frakm):Z(\frakm)\to Z(\frakm)$
	is $R/\frakm$.
}

\begin{prp}
	Let $R$ be a noetherian ring and let $(B,\tau)$ be an Azumaya algebra with involution
	over $R$. 
	Let $d$ be the dimension\footnote{
		Here, the dimension
		of a topological space $X$ is the supremum of the set of integers
		$d\geq 0$ for which there is a chain of irreducible closed subsets
		$X_0\subsetneq X_1\subsetneq \dots\subsetneq X_d$ in $X$.	
	} of $\Max R$, the set of maximal ideals in $R$,
	viewed as a topological subspace of $\Spec R$. If for every $\frakm\in \Max R$,
	$(B(\frakm),\tau(\frakm))$ is not a degree-$2$ central simple algebra
	with a symplectic involution, then $\gen_R(B,\tau)\leq d+1$.
\end{prp}

\begin{proof}
	By   \cite[Thm.~1.2]{First_2017_number_of_generators},
	$\gen_R(B,\tau)\leq d+\sup\{\gen_{R/\frakm}(B(\frakm),\tau(\frakm))\where
	\frakm\in\Max R\}$.
	Since $(B(\frakm),\tau(\frakm))$ is a central simple algebra with involution over $R/\frakm$,
	the required upper bound follows from Proposition~\ref{PR:central-simple}.
\end{proof}

\section{The Irreducible Components of \texorpdfstring{$Z_r$}{Zr} in the Case \texorpdfstring{$A=(A_n,*)$}{A=(An,*)}}
\label{sec:irred-comps}

We finish with describing the irreducible components of $Z_r$
in the case $A=(A_n,*)$. 
Thanks to what we have shown so far, giving a finite list of closed
subsets of $Z_r$ which includes the irreducible components is a simple task,
and the majority of the work will be dedicated to determining which of these sets
is indeed an irreducible component. That turns out to depend on $r$, $n$ and $\Char K$.

Throughout this section, we fix $r,n\in\N$,
let   $K$ be an algebraically closed field and write
$A=(A_n,*)$.
For every $i\in  \{1,\dots,n-1\}$, we denote by 
\[ X_i\subseteq A^r\] 
the set of tuples 
$((a_1,b_1),\dots,(a_r,b_r))\in A^r$
 for which there is an $i$-dimensional subspace
$V\subseteq K^n$ such that $a_1,\dots,a_r$ stabilize $V$
and $b_1,\dots,b_r$ stabilize $V^\perp$. (As before, $V^\perp$ is  w.r.t.\ the 
bilinear pairing $\Trings{x,y}=\sum_j x_j y_j$.)
In addition, let 
\begin{align*}
Y &:= \{((a_1,b_1),\dots,(a_r,b_r))\in A^r\suchthat \text{there is} \\ 
&\qquad \quad  \text{a symmetric non-alternating $p\in \GL_n(K)$ s.t.\ $pa_ip^{-1}=b_i$ for all $i$} \},\\
Y'&:= \{((a_1,b_1),\dots,(a_r,b_r))\in A^r\suchthat \text{there is}  \\
&\qquad \quad  \text{an alternating $p\in \GL_n(K)$ s.t.\ $pa_ip^{-1}=b_i$ for all $i$} \}.
\end{align*}
The set $Y'$ is empty if $n$ is odd. When referring to $\overline{Y}$
and $\overline{Y'}$, the closure is taken in $A^r$ and with respect to the Zariski topology.

\begin{thm}\label{TH:irred-comps}
	With notation as above:	
	\begin{enumerate}[label=(\roman*)]
		\item If $n$ is odd, or $(n,r)=(2,1)$, 
		or $\Char K=2$, 
		then   the irreducible components of $Z_r$  
		are $X_1,\dots,X_{n-1}$ and $\overline{Y}$.
		\item If $n$ is even, $(n,r)\neq (2,1)$ and $\Char K\neq 2$, 
		then   the irreducible components of $Z_r$
		are  $X_1,\dots,X_{n-1},\overline{Y}$ and $\overline{Y'}$.
	\end{enumerate}
\end{thm}

We will prove the theorem in a series of lemmas.
As before, we let $\M=\nMat{K}{n}$, $\GL_n=\GL_n(K)$
and so forth. In addition, let $G$ denote  the identity connected
component of $\Aut_K(A)$, which we identify
with $\PGL$ (Proposition~\ref{PR:aut-group}).  
Recall from   Proposition~\ref{PR:representatives} that  $G$
acts on the maximal subalgebras of $A$ and the orbits 
are represented by  $A_{V(1)},\dots,A_{V(n-1)}$, $B_{[I]}$ if $n$
is odd and $A_{V(1)},\dots,A_{V(n-1)}$, $B_{[I]}$, $B_{[\Omega]}$ if $n$
is even. 

\begin{lem}\label{LM:properties-of-X-Y}
	The following holds:
	\begin{enumerate}[label=(\roman*)]
		\item $X_i=\bigcup_{g\in G} g(A_{V(i)})^r$ 		
		for all $i\in\{1,\dots,n-1\}$. 
		Moreover, $X_i$ is closed in $A^r$ and $\dim X_i = 2rn^2 - (2r-1) (n-i)i $.
		\item $Y = \bigcup_{g\in G}g(B_{[I]})^r$ and
		$\dim \overline{Y} \leq (r+1)n^2-\frac{n(n-1)}{2}-1$. When $\Char K\neq 2$, equality
		holds.
		\item $Y' = \bigcup_{g\in G}g(B_{[\Omega]})^r$ and
		$\dim \overline{Y'}=(r+1)n^2-\frac{n(n+1)}{2}-1$, provided $n$ is even.
	\end{enumerate}
\end{lem}

\begin{proof}
	That $X_i=\bigcup_{g\in G} g(A_{V(i)})^r$ follows immediately from Lemma~\ref{LM:action}.
	The same lemma tells us that 
	\[\bigcup_{g\in G} g(B_{[I]})^r=\bigcup_{p\in\GL_n: p\sim I} \{((a_1,pa_1p^{-1}),\dots,
	(a_r,pa_rp^{-1}))\where a_1,\dots,a_r\in \M\},\] 
	($p\sim I$ means that $p $ is congruent to $I$).
	Since $K$ is algebraically closed, a matrix is congruent to $I$ if and only
	if and only it is symmetric and not alternating, so $\bigcup_{g\in G} g(B_{[I]})^r=Y$.
	A similar argument shows that $Y' = \bigcup_{g\in G}g(B_{[\Omega]})^r$, because a matrix $p\in \GL_n$ is congruent
	to $\Omega$ if and only if $p$ is alternating.
	
	We turn to prove that $X_i$ is closed in $A^r$.
	This follows from Proposition~\ref{PR:dim-Zr-detailed}(v) because
	$\Stab_G(A_{V(i)})=\PS_{V(i)}^\times$ (Proposition~\ref{PR:stabilizers})
	and $\PS_{V(i)}^\times$
	is a parabolic subgroup of $\PGL$. (Indeed, $\PGL/\PS_{V(i)}^\times\cong \GL_n/S_{V(i)}^\times$
	is the Grassmannian   of $i$-spaces inside an $n$-space, which is a complete variety.)
	
	We finish by computing (or bounding)
	the dimensions of $X_i$, $\overline{Y}$ and $\overline{Y'}$ using Proposition~\ref{PR:dim-Zr-detailed}(ii).
	Since $\gen_K(A_{V(i)})=1\leq r$
	(Proposition~\ref{PR:subalgebra-generators}),
	we have $\dim X_i = \dim G+r\dim A_{V(i)}-\dim \Stab_G(A_{V(i)})$.
	The quantity $ r\dim A_{V(i)}-\dim \Stab_G(A_{V(i)})$ is the number $s(i,r)$
	computed in the proof of Theorem~\ref{TH:codim-unitary}, so
	$\dim X_i=s(i,r)+n^2-1= 2rn^2 - (2r-1) (n-i)i$.
	In the same way, we have 
	\[\dim \overline{Y} \leq \dim G+s'(i,r)=  (r+1)n^2-\frac{n(n-1)}{2}-1 ,\]
	with equality holding when $\Char K\neq 2$,
	and when $n$ is even,
	\[\dim \overline{Y'} = \dim G+s''(i,r)=  (r+1)n^2-\frac{n(n+1)}{2}-1 ,\]
	provided $\gen(B_{[\Omega]})\geq r$. The latter condition holds by Proposition~\ref{PR:subalgebra-generators}
	as long as $(n,r)\neq (2,1)$, so it remains to find $\dim \overline{Y'}$
	when $(n,r)=(2,1)$.
	
	In the exceptional case $(n,r)=(2,1)$,
	all alternating matrices $p\in \GL_2$ are proportional to $\Omega$,
	so $Y'=B_{[\Omega]}$. In particular, $Y'$ is closed
	and $\dim Y'=4$, which happens to coincide with  $(r+1)n^2-\frac{n(n+1)}{2}-1$.
\end{proof}

\begin{remark}\label{RM:not-closed}
	In general, the sets $Y$ and $Y'$ are not closed in   $A^r$.
	Indeed, suppose that $n\geq 2$, let
	$\alpha\in K-\{1\}$ and put $a_\alpha=[\begin{smallmatrix}
	\alpha & 1 \\
	0 & 1\end{smallmatrix}]\oplus I_{n-2}$,
	$b_\alpha=[\begin{smallmatrix}
	1 & 0 \\
	0 & \alpha \end{smallmatrix}]\oplus I_{n-2}$
	and $p_\alpha=[\begin{smallmatrix}
	0 & \alpha-1 \\
	\alpha-1 & 1\end{smallmatrix}]\oplus I_{n-2}$. 
	One readily checks that $p_\alpha a_\alpha p_\alpha^{-1}=b_\alpha$
	and therefore $((a_\alpha,b_\alpha),\dots,(a_\alpha,b_\alpha))\in Y$
	for all $\alpha\neq 1$.
	This means that $((a_1,b_1),\dots,(a_1,b_1))\in \overline{Y}$,
	but $((a_1,b_1),\dots,(a_1,b_1))$ cannot be in $Y$ because $a_1$
	is not conjugate to $b_1$.
	
	A similar argument using $\frac{n}{2}$ in place of $n$
	and the  matrices $I_2\otimes a_\alpha$, $I_2\otimes b_\alpha$,
	$[\begin{smallmatrix}
	0 & -1 \\
	1 & 0\end{smallmatrix}]\otimes p_\alpha$ shows that
	$Y'$ is not closed in $A^r$ when $n$ is even and $n\geq 4$.
	On the other hand, $Y$ is closed when $n=1$, because then   $Y=B_{[I]}^r$,
	and $Y'$ is closed when $n=2$, because then $Y'=B_{[\Omega]}^r$. 
\end{remark}

\begin{lem}\label{LM:maximal-almost}
	If $n$ is odd (resp.\ even), then the irreducible components of $Z_r$
	are the maximal members of $X_1,\dots,X_{n-1},\overline{Y}$
	(resp.\ $X_1,\dots,X_{n-1},\overline{Y},\overline{Y'}$).
\end{lem}

\begin{proof}
	This follows from Proposition~\ref{PR:dim-Zr-detailed}(iii),
	Proposition~\ref{PR:representatives} and Lemma~\ref{LM:properties-of-X-Y}.
\end{proof}

The maximal members of  
$X_1,\dots,X_{n-1},
\overline{Y},\overline{Y'}$ will be determined in the following lemmas.

\begin{lem}\label{LM:not-contain-easy}
	With notation as above:
	\begin{enumerate}[label=(\roman*)]
		\item No two of the sets $X_1,\dots,X_{n-1},\overline{Y}$ contain each other.
	\end{enumerate}
	Furthermore, when $n$ is even:
	\begin{enumerate}[label=(\roman*),resume]
		\item None of $X_1,\dots,X_{n-1},\overline{Y}$ is contained in $\overline{Y'}$.
		\item $\overline{Y'}\nsubseteq X_1,\dots,X_{n-1}$, provided  $(n,r)\neq (2,1)$.
		\item $\overline{Y'}\nsubseteq \overline{Y}$, provided $(n,r)\neq (2,1)$ and $\Char K\neq 2$.
	\end{enumerate}
\end{lem}

\begin{proof} 
	Let $i,j  \in\{1,\dots,n-1\}$.
	Each of the algebras $A_{V(1)},\dots,A_{V(n-1)}$, $B_{[I]}$
	can be generated by $1$ element by Proposition~\ref{PR:subalgebra-generators}.
    Since $X_i$ is closed (Lemma~\ref{LM:properties-of-X-Y}),
	Proposition~\ref{PR:dim-Zr-detailed}(iv) tells
	us that $\overline{Y}\nsubseteq X_i $ and $X_i\nsubseteq X_j$ for $i\neq j$.
	The same argument shows that $\overline{Y'}\nsubseteq X_i$
	when $n$ is even and $(n,r)\neq (2,1)$.
	
	The   assertions $X_i\nsubseteq \overline{Y}$ and $X_i\nsubseteq \overline{Y'}$
	hold because the varieties on left hand sides have larger dimension than those
	on the right (Lemma~\ref{LM:properties-of-X-Y}). This also shows that
	$\overline{Y}\nsubseteq \overline{Y'}$ when $\Char K\neq 2$.
	
	It remains to show that 
	$\overline{Y}\nsubseteq \overline{Y'}$ when $\Char K=2$,
	and  	
	$\overline{Y'}\nsubseteq \overline{Y}$ when
	$(n,r)\neq (2,1)$ and $\Char K\neq 2$.

	We begin with showing the second statement, i.e., (iv). 
	To that end, we apply Proposition~\ref{PR:Grassmannian-test}
 	with $A_1=B_{[\Omega]}$, $A_2=B_{[I]}$ and  $G=\PGL$.
 	Note that
	$\gen_K(A_1)\geq r$  by Proposition~\ref{PR:subalgebra-generators}, because $(n,r)\neq (2,1)$.
	
	Let $S$, resp.\ $S'$,  denote set of elements in $\PGL$
	represented by a symmetric, resp.\ alternating,
	matrix in $\GL_n$. 
	By thinking of $\PGL$ as an open subvariety of $\mathbb{P}(\M)$,
	one sees that $S$ and $S'$ are closed in $\PGL$.
	Write $m=n^2$ and let $X$ be the Grassmannian of $m$-spaces in $A=\M\times \M$.
	Then the subsets $Z_1,Z_2\subseteq X$ from Proposition~\ref{PR:Grassmannian-test}
	are $Z_1=\{B_{[p]}\where p\in  S'\}$
	and $Z_2=\{B_{[p]}\where p\in S\}$. 
	By that proposition, in order to show that
	$\overline{Y'}\nsubseteq \overline{Y }$, it is enough
	to show that $\overline{Z_1}\nsubseteq \overline{Z_2}$ in $X$.
	
	Let $H$ be the $K$-vector space
	of linear endomorphisms of $\M$.
	The map $\vphi:H\to X$ sending $t\in H$ to its graph,
	$\{(a,t(a))\where a\in \M\}$, 
	defines an isomorphism from $H$
	onto the open subset $U$ of $X$ consisting of $K$-subspaces
	$L\subseteq \M\times \M$ satisfying $L\oplus (0\times \M)=\M\times \M$.
	Since every $K$-algebra endomorphism of $\M$ is an automorphism
	(because $\M$ is simple and finite dimensional),
	$\PGL$ conincides with the $K$-algebra endomorphisms of $\M$,
	which form a closed subvariety of $H$.
	We may therefore view $S$ and $S'$ as closed subvarieties of $H$.
	Now observe that $Z_1=\vphi(S')$ while $Z_2=\vphi(S)$ (because $\Char K\neq 2$).
	Since $\vphi:H\to U$ is an isomorphism of $K$-varieties, it follows 
	that $\vphi^{-1}(\overline{Z_1})=\vphi^{-1}(\mathrm{Cl}_U(Z_1))=\mathrm{Cl}_H(S')=S'$,
	and similarly, $\vphi^{-1}(\overline{Z_2})=S$.
	Since $S'\nsubseteq S $ (because $\Char K\neq 2$), we must have $\overline{Z_1}\nsubseteq \overline{Z_2}$,
	and we conclude that $\overline{Y'}\nsubseteq \overline{Y }$.
	
	Showing that $\overline{Y }\nsubseteq \overline{Y'}$
	when $\Char K= 2$ is similar, but with the following differences:
	We take $A_1=B_{[I]}$ and $A_2=B_{[\Omega]}$ in Proposition~\ref{PR:Grassmannian-test},
	and since $\Char K = 2$, we get $Z_1=\vphi(S-S')$
	and $Z_2=\vphi(S')$.
	The desired conclusion $\overline{Z_1}\nsubseteq \overline{Z_2}$
	follows from $\overline{S-S'}=S\nsubseteq S'$.
\end{proof}

When $\Char K=2$ or $(n,r)=(2,1)$, it can
happen that   $\overline{Y'}\subseteq X_i$ or $\overline{Y'}\subseteq \overline{Y}$.

\begin{lem}\label{LM:do-contain}
	Suppose $n$ is even.
	\begin{enumerate}[label=(\roman*)]
		\item[(i)] If $\Char K=2$, then $\overline{Y'}\subseteq \overline{Y}$.
		\item[(ii)] If $(n,r)=(2,1)$, then $\overline{Y'}\subseteq X_1\cap \overline{Y}$.
	\end{enumerate}
\end{lem}

\begin{proof}
	(i) When $\Char K=2$, the space   of  alternating $n\times n$ matrices
	is strictly contained in the space of  symmetric matrices. This implies
	readily that $Y'\subseteq \overline{Y}$, so $\overline{Y'}\subseteq\overline{Y}$.
	
	(ii) Suppose $(n,r)=(2,1)$. We first prove that $\overline{Y'}\subseteq X_1$.
	As noted in the last paragraph of the proof of Lemma~\ref{LM:properties-of-X-Y},
	we   have $Y'=\overline{Y'}=B_{[\Omega]}$. 
	Let $
	x=([\begin{smallmatrix} a & b \\ c & d \end{smallmatrix}],
	[\begin{smallmatrix} d & -c \\ -b & a \end{smallmatrix}])\in B_{[\Omega]}$.
	Since $K$ is algebraically closed, $[\begin{smallmatrix} a & b \\ c & d \end{smallmatrix}]$
	has an eigenvector $v:=[\begin{smallmatrix} s \\ t \end{smallmatrix}]\in K^2$.
	One readily checks that $v^\perp$ is spanned by
	$v':=[\begin{smallmatrix} -t \\ s \end{smallmatrix}]$
	and $v'$ is an eigenvector of $[\begin{smallmatrix} d & -c \\ -b & a \end{smallmatrix}]$.
	It follows that $x\in S_{Kv}\times S_{(Kv)^{\perp}}\subseteq X_1$.
	
	We turn to prove that $\overline{Y'}\subseteq \overline{Y}$. We already showed it
	when $\Char K=2$, so assume $\Char K\neq 2$.
	It is enough to prove that a dense subset of $Y'=B_{\Omega}$
	is contained in $Y$.
	Again, let $ 
	x=([\begin{smallmatrix} a & b \\ c & d \end{smallmatrix}],
	[\begin{smallmatrix} d & -c \\ -b & a \end{smallmatrix}])\in  B_{[\Omega]}$
	and, given $\gamma\in K$, 
	put $u_\gamma = [\begin{smallmatrix} a & b \\ c & d \end{smallmatrix}]+\gamma I $.
	Since $
	\Omega 	
	[\begin{smallmatrix} a & b \\ c & d \end{smallmatrix}]
	\Omega^{-1}=
	[\begin{smallmatrix} d & -c \\ -b & a \end{smallmatrix}]$,
	we also have $(\Omega u_\gamma) 	
	[\begin{smallmatrix} a & b \\ c & d \end{smallmatrix}]
	(\Omega u_\gamma)^{-1}=
	[\begin{smallmatrix} d & -c \\ -b & a \end{smallmatrix}]$, provided $u_\gamma\in \GL_2$.
	The set of $x\in B_{[\Omega]}$
	for which $u_{-\frac{a+d}{2}}\in \GL_2$ is dense in $B_{[\Omega]}$. Since $\Omega u_{-\frac{a+d}{2}}$ is symmetric,
	all those  $x$-s satisfy $x\in Y$. We conclude that $Y$ contains a dense subset of $Y'$.
\end{proof}

We can now conclude the proof of Theorem~\ref{TH:irred-comps}.

\begin{proof}[Proof of Theorem~\ref{TH:irred-comps}]
	Suppose $n$ is odd. By Lemma~\ref{LM:maximal-almost},
	the irreducible components of $Z_r$ are the maximal members of 
	$X_1,\dots,X_{n-1}$, $\overline{Y}$, and none of these sets is contained
	in any of the others by  Lemma~\ref{LM:not-contain-easy}(i).
	
	Suppose $n$ is even. By Lemma~\ref{LM:maximal-almost},
	the irreducible components of $Z_r$ are the maximal members of 
	$X_1,\dots,X_{n-1}$, $\overline{Y}$, $\overline{Y'}$.
	If $\Char K=2$ or $(n,r)=(2,1)$, then $\overline{Y'}\subseteq \overline{Y}$
	(Lemma~\ref{LM:do-contain})
	and we are done by Lemma~\ref{LM:not-contain-easy}(i).
	If $\Char K\neq 2$ and $(n,r)\neq (2,1)$, then
	Lemma~\ref{LM:not-contain-easy}(i)--(iv) tells
	us that $X_1,\dots,X_{n-1}$, $\overline{Y}$, $\overline{Y'}$ are all maximal.
\end{proof}

\bibliographystyle{plain}
\bibliography{MyBib_18_05}

\end{document}

%% file: amsart_header_18_05.tex
\usepackage{amsfonts}
\usepackage{amssymb}
\usepackage{amsmath}
\usepackage{hyperref}
\usepackage{mathrsfs}
\usepackage{centernot}
\usepackage{mathdots}
\usepackage{stmaryrd}
\usepackage[all]{xy}

\usepackage{mathtools}

\usepackage{color}

\newtheorem{thm}{Theorem}[section]
\newtheorem{lem}[thm]{Lemma}
\newtheorem{prp}[thm]{Proposition}
\newtheorem{cor}[thm]{Corollary}

\theoremstyle{definition}

\newtheorem{remark}[thm]{Remark}

\theoremstyle{plain}

\newcommand{\rem}[1]{}

\newcommand{\C}{\mathbb{C}}

\newcommand{\N}{\mathbb{N}}

\newcommand{\HH}{{\mathrm{H}}}

\newcommand{\frakm}{{\mathfrak{m}}}

\newcommand{\calO}{{\mathcal{O}}}

\newcommand{\bbP}{{\mathbb{P}}}

\newcommand{\veps}{\varepsilon}
\newcommand{\vphi}{\varphi}

\newcommand{\suchthat}{\,:\,}
\newcommand{\where}{\,|\,}

\newcommand{\quo}[1]{\overline{#1}}

\newcommand{\Trings}[1]{\left< #1 \right>}

\newcommand{\floor}[1]{\lfloor {#1} \rfloor}
\newcommand{\Floor}[1]{\left\lfloor {#1} \right\rfloor}

\newcommand{\Ceil}[1]{\left\lceil {#1} \right\rceil}

\DeclareMathOperator{\Aut}{Aut} %
\DeclareMathOperator{\Cent}{Cent} %
\DeclareMathOperator{\Char}{char} %
\DeclareMathOperator{\End}{End} %
\DeclareMathOperator{\GL}{GL} %
\DeclareMathOperator{\id}{id} %

\DeclareMathOperator{\Max}{Max}
\DeclareMathOperator{\Nrd}{Nrd} %
\newcommand{\op}{\mathrm{op}} %
\DeclareMathOperator{\rank}{rank}
\DeclareMathOperator{\Spec}{Spec} %
\DeclareMathOperator{\Stab}{Stab} %
\DeclareMathOperator{\Tr}{Tr} %
\DeclareMathOperator{\Trd}{Trd} %

\newcommand{\nGL}[2]{\mathrm{GL}_{#2}({#1})}

\newcommand{\nMat}[2]{\mathrm{M}_{#2}(#1)}

\newcommand{\trans}{{\mathrm{t}}}

